\documentclass[reqno]{amsart}
\usepackage{mathrsfs,amssymb,mathrsfs, multirow,setspace}
\usepackage{amsmath}
\usepackage{tabularray}
\UseTblrLibrary{amsmath}
\usepackage[a4paper, total={7in, 9in}]{geometry}
\usepackage{textpos,xcolor,tikz,eso-pic,multicol,verbatim}
\usepackage{enumerate}
\usepackage{geometry,mathtools}
\theoremstyle{plain}
\usepackage{graphicx}
\usepackage{mathtools}
\usepackage[utf8]{inputenc}

\newtheorem{theorem}{Theorem}[section]
\newtheorem{lemma}[theorem]{Lemma}

\newtheorem{corollary}[theorem]{Corollary}

\theoremstyle{definition}

\theoremstyle{remark}
\newtheorem{remark}[theorem]{Remark}

\input xy
\xyoption{all}

\renewcommand{\Bbb}{\mathbb} 

\begin{document}
	
	\title{Sombor Spectrum of Super Graphs defined on groups}
        \author[ Ekta Pachar, Sandeep Dalal, Jitender Kumar]{Ekta Pachar$^1$, Sandeep Dalal$^2$, Jitender Kumar$^{1*}$}

 \address{1 - Department of Mathematics, Birla Institute of Technology and Science Pilani, Pilani-333031, India\\
 2 - Department of Mathematics, Punjab Engineering College (Deemed to be University), Chandigarh- 160012, India}
 \email{ektasangwan0@gmail.com, deepdalal10@gmail.com, jitenderarora09@gmail.com }


\maketitle


\section*{abstract}

Given a simple graph $A$ on a group $G$ and an equivalence relation $B$ on $G$, the $B$ super $A$ graph is defined as a simple graph, whose vertex set is $G$ and two vertices $g$, $h$ are adjacent if either they are in the same equivalence class or there exist $g^{\prime} \in[g]$ and $h^{\prime} \in[h]$ such that $g^{\prime}$ and $h^{\prime}$ are adjacent in $A$. In the literature, the $B$ super $A$ graphs have been investigated by considering $A$ to be either power graph, enhanced power graph, or commuting graph and $B$ to be an  equality, order or conjugacy relation. In this paper, we investigate the Sombor spectrums of these $B$ super $A$ graphs for certain non-abelian groups, viz. the dihedral group, generalized quaternion group and the semidihedral group, respectively.


\section{introduction}
The study of algebraic structures using graph properties has been a significant area of research over the past
three decades. There are several interesting papers connecting the research in graph theory with algebraic structure, viz. non-cyclic graphs \cite{abdollahi2009noncyclic}, Cayley graphs \cite{a.kelarev2002undirected}, commuting graphs \cite{a.segev2001commuting}, power graphs \cite{chakrabarty2009undirected}, and enhanced power graphs \cite{bera2018enhanced}. 
Algebraic graphs over particular finite groups are well-studied. For instance the power graphs, enhanced power graphs, commuting graphs etc.,  over dihedral groups $D_{2n}$, semidihedral groups $SD_{8n}$, generalized quaternion groups $Q_{4n}$, and finite cyclic groups $\mathbb Z_n$ (see \cite{ali2016commuting,a.ashrafi2017automorphism,a.chattopadhyayspectralradius2017,a.chattopadhyay-laplacian,a.2017Hamzaeh_automorphism, dalal2021semidihedral,rajkumar2021laplacian}). In continuation of this study, the notion of super graphs over finite groups was  introduced by Arun Kumar \emph{et al.} \cite{arunkumar2024graphs}. Let $B$ be an equivalence relation on a finite group $G$. For $g \in G$, $[g]$ be an equivalence class of $g$ in $G$. The $B$ super $A$ graph is defined as a simple graph, whose vertex set is $G$ and two vertices $g$, $h$ are adjacent if either they are in the same equivalence class or there exist $g^{\prime} \in[g]$ and $h^{\prime} \in[h]$ such that $g^{\prime}$ and $h^{\prime}$ are adjacent in $A$. Moreover, the subgaph induced by the vertices belongs the equivalence class in $[x]_B$  in the $B$ super $A$ graph is complete. In this article, we study the following three types of graphs and three types of equivalence relations.\\
Three graphs over a finite group $G$:

\begin{enumerate}
\item[\rm (i)] The \emph{power graph} $\mathcal{P}(G)$ of a group $G$ is an undirected simple graph whose vertex set is $G$ and two vertices $x$ and $y$ are adjacent if either $x \in \langle y \rangle$ or $y \in \langle x \rangle$.

\item[\rm (ii)] The \emph{enhanced power graph} $\mathcal{P}_E(G)$ of a group $G$ is an undirected simple graph whose vertex set is $G$ and two vertices $x$ and $y$ are adjacent if both $x$ and $y$ belongs to a  same cyclic subgroup of a group $G$.

\item[\rm (iii)] The \emph{commuting  graph} $\Delta(G)$ of a group $G$ is an undirected simple graph whose vertex set is $G$ and two vertices $x$ and $y$ are adjacent whenever $xy = yx$.
\end{enumerate}

Three equivalence relations on a group $G$ are:

\begin{enumerate}
	\item[\rm (a)]  \emph{equality relation} $(x, y) \in B_{_E}$ if and only if $x = y$; 
	
	\item[\rm (b)] \emph{order relation} $(x, y) \in B_o$ if and only if $ o (x) = o(y)$, where $o(a)$ denotes the order of $a \in G$; 
	
	\item[\rm (c)] \emph{conjugacy relation} $(x, y) \in B_c$ if and only if $x = aya^{-1}$ for some $a \in G$. 
\end{enumerate}

Note that the equality super $\Gamma$ graph is equal to the same graph $\Gamma$. In this article, we denote the order super $\Gamma(G)$ graph and the conjugacy super $\Gamma (G)$ graph by $\Gamma^o(G)$ and $\Gamma^c(G)$, respectively, where $\Gamma(G) \in  \{\mathcal{P}(G),  \mathcal{P}_E(G),  \Delta(G)\}$. For some pairs of graphs, Kumar \emph{et al.} \cite{arunkumar2022super} characterized finite groups $G$ such that two graphs in a particular pair are equal. Moreover, Kumar \emph{et al.} \cite{arunkumar2024main} obtained the spectrum of equality super commuting and conjugacy super commuting graphs for the dihedral groups and the generalized quaternion groups and show that these are not integral. Further, Dalal \emph{et al.}  classified the finite groups $G$ such that two graph in the pair $\mathfrak{P}$ are equal, where $\mathfrak{P} \in \{ \mathfrak{P}_1, \mathfrak{P} _2, \mathfrak{P} _3, \mathfrak{P} _4, \mathfrak{P} _5, \mathfrak{P} _6 \}$ with $\mathfrak{P}_1 = \{ \mathcal{P}(G),   \mathcal{P}^o(G)\} $,  $\mathfrak{P}_2 = \{ \mathcal{P}^o(G),   \mathcal{P}^o_E(G)\} $, $\mathfrak{P}_3 = \{ \mathcal{P}^c(G),   \mathcal{P}^c_E(G)\} $,  $\mathfrak{P}_4 = \{ \mathcal{P}^o(G),   \Delta^o(G)\} $,  $\mathfrak{P}_5 = \{ \Delta(G),   \Delta^o(G)\} $,  $\mathfrak{P}_6 = \{ \mathcal{P}_E(G),   \mathcal{P}_E^o(G)\} $. Finally, they proved that the diameter of the reduced order super commuting graph $\Delta^o(G)^*$, where $G \in \{S_n, A_n\}$,  is either two or three and they posed a conjecture that $\Delta^o(G)^* = 3$, where $G \in \{S_n, A_n\}$. Bra$\stackrel{\vee}{c}$i$\stackrel{\vee}{c}$ \emph{et al.} \cite{bravcivc2025diameter} proved that $\Delta^o(S_n)^* = 3$ and $\Delta^o(A_n)^* = 3$.  The study of graph invariants such as Laplacian spectrum,  Sombor spectrum, metric dimension, and detour distance, is both intriguing and significant due to their valuable applications.  Various authors have studied the Laplacian spectrum of certain graphs on algebraic structures (see \cite{chattopadhyay2015laplacian,dalal2021semidihedral,rajkumar2021laplacian}). Dalal \emph{et al.} obtained the adjacency spectrum and Laplacian spectrum of conjugacy super commuting graphs and order super commuting graphs of dihedral group $D_{2n} (n \geq 3)$, generalized quaternion group $Q_{4m} (m \geq 2)$ and the non-abelian group $\mathbb Z_p \rtimes \mathbb Z_q$ of order $pq$,  where $p$ and $q$ are distinct primes with $q \mid (p-1)$. The notion of Sombor spectrum of a graph was introduced by Gutman \cite{gutman2021geometric}. Rather \emph{et al.} \cite{rather2024comaximal} obtained the sharp bounds for the Sombor index of comaximal graphs of the commutative rings $\mathbb{Z}_n$. They also found the Sombor eigenvalues and bounds for the Sombor energy of comaximal graphs of the ring $\mathbb{Z}_n$. Moreover, Anwar \emph{et al.} \cite{anwar2024Sombor} investigated the Sombor spectrum of cozero divisor graph of ring $\mathbb Z_n$. Motivated with the work on Sombor spectrum of certain algebraic graphs, in this paper, we aim to investigate the Sombor spectrum of super graphs defined on non-abelian groups.

The structure of the paper is as follows. In Section $2$, we recall the concepts of group theory and graph theory and review some of the well-known findings pertaining to Sombor spectrum. The characteristic polynomial of the Sombor matrix of $\mathcal{R}$-super $\Gamma$ graph $\Gamma^{\mathcal{R}}$ is discussed in Section $3$. For $G \in \{D_{2n}, Q_{4n}, SD_{8n}\}$, the Sombor spectra of $\Delta(G), \Delta^o(G)$, and $\Delta^c(G)$ have been obtained in Section $4$. The Sombor spectra of $\mathcal{P}_E(G), \mathcal{P}_E^o(G)$, and $\mathcal{P}_E^c(G)$ obtained in Section $5$. The Sombor spectra of $\mathcal{P}(G)$, $\mathcal{P}^o(G)$, and $\mathcal{P}^c(G)$ have been investigated in Section $6$.


\section{preliminaries}
We recall necessary definitions, results and notations of graph theory from \cite{b.West}. A graph $\Gamma$ is an ordered pair $ \Gamma = (V, E)$, where $V = V(\Gamma)$ denotes the set of vertices and $E = E(\Gamma)$ denotes the set of edges in $\Gamma$. We say that two distinct vertices $a$ and $b$ are $\mathit{adjacent}$, denoted by $a \sim b$, if there is an edge connecting $a$ and $b$.  The \emph{neighbourhood} $N(x)$ of a vertex $x$ is the collection of all vertices which are adjacent to $x$ in the graph $ \Gamma $. Additionally, we denote $N[x] = N(x) \cup \{x\}$. We are considering simple graphs, i.e. undirected graphs with no loops or repeated edges. A \emph{subgraph} of a graph $\Gamma$ is defined as a graph $\Gamma'$ for which the vertex set $V(\Gamma') \subseteq V(\Gamma)$ and $E(\Gamma') \subseteq E(\Gamma)$. The subgraph $\Gamma(X)$ of a graph $\Gamma$, induced by a set $X$, consists of the vertex set $X$ where two vertices are connected by an edge if and only if they are adjacent in $\Gamma$. A graph $\Gamma$ is said to be a \emph{complete graph} if every pair of distinct vertices are adjacent. We denote $K_n$ by the complete graph on $n$ vertices.
A subgraph $\Gamma'$ of a graph $\Gamma$ is said to be \emph{spanning  subgraph} if $V(\Gamma) = V(\Gamma')$ and $E(\Gamma') \subseteq E(\Gamma)$. 
For two graphs $\Gamma_1$ and $\Gamma_2$ with disjoint vertex sets, the \emph{join graph} $\Gamma_1 \vee \Gamma_2$ of $\Gamma_1$ and $\Gamma_2$ whose vertex set is $V(\Gamma_1) \cup V(\Gamma_2)$ and $E(\Gamma_1 \vee \Gamma_2) = E(\Gamma_1) \cup E(\Gamma_2) \cup \{ (x, y) : x \in V(\Gamma_1)~ \text{and}~  \in V(\Gamma_2)\}$. Let $\Gamma$ be a graph on $k$ vertices and $V(\Gamma) =\{u_1, u_2, \ldots, u_k\}$. Suppose that $\Gamma_1, \Gamma_2, \ldots, \Gamma_k$ are $k$ pairwise disjoint graphs. Then \emph{generalised join graph} $\Gamma[\Gamma_1, \Gamma_2, \ldots, \Gamma_k]$ of $\Gamma_1, \Gamma_2, \ldots, \Gamma_k$ is the graph formed by replacing each vertex $u_i$ of $\Gamma$ by $\Gamma_i$ and then joining each vertex of $\Gamma_i$ to every vertex of $\Gamma_j$ whenever $u_i \sim u_j$ in $\Gamma$. Let $\Gamma$ be a finite simple undirected graph with a set of vertex $V(\Gamma) = \{u_1, u_2, \ldots, u_n\}$. For a graph $\Gamma$, the Sombor matrix $S(\Gamma)$ is defined as 
\[\
S(\Gamma)= (s_{ij})= \begin{cases}
   \sqrt{(deg(u_i))^2 + (deg(u_j))^2},& ~ \text{if}~ u_i \sim u_j; \\
   0,& \text{otherwise}.    
\end{cases}
\]

Sombor matrix is real and symmetric. The spectrum of the Sombor matrix is known as the Sombor spectrum for the graph $\Gamma.$  The eigenvalues of $S(\Gamma)$, known as the \emph{Sombor eigenvalues} of the graph $\Gamma$, are denoted by $\lambda_1(\Gamma),\;\lambda_2(\Gamma) ,\cdots,\lambda_n(\Gamma)$. Let us denote the distinct eigenvalues of $\Gamma$ by $\lambda_{n_1}(\Gamma),\;  \lambda_{n_2}(\Gamma),\;  \cdots  \lambda_{n_r}(\Gamma)$ with multiplicities $m_1, m_2, \ldots, m_r$, respectively. The \emph{Sombor spectrum} $\sigma (S(\Gamma))$ of $\Gamma$ is denoted by $\displaystyle \begin{pmatrix}
\lambda_{n_1}(\Gamma) & \lambda_{n_2}(\Gamma) & \cdots& \lambda_{n_r}(\Gamma)\\
 m_1 & m_2 & \cdots & m_r
\end{pmatrix}$.

The following result gives us a relation between the clique, independent set, and  Sombor eigenvalues of a graph $\Gamma$.

\begin{lemma}\label{multiplicity}\cite{pirzada2025spectrum}
Let $\Gamma$ be a connected graph with $n$ vertices and let $S= \{u_1, u_2,\ldots,u_t \}$
be a set of vertices in $\Gamma$ such that $N(u_i)\setminus S = N(u_j )\setminus S$ for each
$1 \leq i, j \leq t$. Then the following hold:
\begin{itemize}
\item[(i)] If $S$ is an independent set, then $0$ is the Sombor eigenvalue of $G$ with multiplicity at least $t -1$.
\item[(ii)] If $S$ is a clique, then $ -d\sqrt{2}$ is the Sombor eigenvalue of $G$ with multiplicity at
least $t-1$, where $d$ is the degree of $u_i$.
\end{itemize}
\end{lemma}

We denote the square matrices of size $n$,  $J_n$ in which each entry is one, $\mathcal{O}_n$ represents the zero matrix and $I_n$ is the identity matrix. Consider an $n \times n$ matrix
\[M =  
 \displaystyle \begin{bmatrix}
 
A_{1,1} &A_{1,2} &\cdots& A_{1,s-1}& A_{1,s}\\
A_{2,1} & A_{2,2} &\cdots&  A_{2,s-1} & A_{2,s}\\
\vdots & \vdots &  \vdots & \vdots & \vdots &  \\
\vdots & \vdots &  \vdots & \vdots & \vdots &  \\
\vdots & \vdots &  \vdots & \vdots & \vdots &  \\ 
A_{s-1,1} & A_{s-1,2} &\cdots& A_{s-1,s-1}& A_{s-1,s}\\
A_{s,1} &A_{s,2} &\cdots& A_{s,s-1} &A_{s,s}\\
\end{bmatrix},\]\\
whose rows and columns are partitioned according to a partition $P =
\{P_1, P_2,\ldots,P_s\}$ of the set $X = \{1, 2,\ldots,n\}$. The quotient matrix $Q = (q_{ij} )$ (see \cite{brouwer2011spectra}) is an $s \times s$ matrix, where $q_{ij}$ th entry is the average row (column) sum of the
block $A_{ij}$ of $M$. The partition $P$ is said to be equitable, if row (column) sum of each
block $A_{i,j}$ is some constant and in such case $Q$ is known as the equitable quotient
matrix.

The next result gives a relation between the eigenvalues of $M$ and the eigenvalues of $Q$.

\begin{theorem}{\cite[Theorem 3.1]{equitablematrix}}\label{equitable matrix}
Let $M$ be an $n \times n$ matrix such that $M_{ij} = s_{ij} J_{n_i,n_j}$ for $i \neq j$, and $M_{ii} = s_{ii} J_{n_i,n_i} + p_i I_{n_i}$. Then the equitable quotient matrix of $M$ is $B = (b_{ij})$ with $b_{ij} = s_{ij} n_j$ if $i \neq j$, and $b_{ii} = s_{ii} n_i + p_i$. Moreover, 
$\sigma(M) = \sigma(B) \cup \{ p_1 ^{[n_1 - 1]}, \dots, p_t ^{[n_t - 1]}\}$.

\end{theorem}
\begin{theorem} \label{equitable}\cite{brouwer2011spectra}
    Let $M$ be an $n \times n$ matrix and $Q$ be its quotient matrix.
Then the following results hold:
\begin{itemize}
    \item[(i)]  If the partition $P$ of the set  $X$ of matrix $M$ is not equitable, then the eigenvalues of $Q$ interlace the eigenvalues of $M$.
\item[(ii)] If the partition $P$ of the set $X$ of matrix $M$ is equitable, then each of the eigenvalue of
$Q$ is the eigenvalue of $M$.
\end{itemize}
\end{theorem}


We end this section by recalling the structure of the following three non-abelian groups

For $n \geq 3$, the \emph{dihedral group} $D_{2n}$ is a group of order $2n$ defined by 
\[
D_{2n} = \langle a, b : a^{n} = b^2 = e, \ ba = a^{-1}b \rangle.
\]
Every element of $D_{2n} {\setminus} \langle a \rangle$ is known to be of the form $a^ib$ for some $0 \leq i \leq n-1$, and it follows that $ \langle a^ib \rangle =   \{ e, a^ib \}  $. Consequently, we have 

\begin{equation}\label{eq(1)}
D_{2n} = \langle a \rangle \cup \bigcup\limits_{ i=  0}^{n-1} \langle a^ib \rangle.
\end{equation}

For $n \geq 2$, the \emph{generalized quaternion group} $Q_{4n}$ is a group of order $4n$ defined by
\[
Q_{4n} = \langle a, b : a^{2n} = e, a^{n} = b^2, ba = a^{-1}b \rangle.
\]
Observe that every element in $Q_{4n} \setminus \langle a \rangle$ can be written as $a^ib$ for some $1 \leq i \leq 2n-1$. Also, $ \langle a^ib \rangle =  \langle a^{n+i}b \rangle =  \{ e, a^ib,a^n,  a^{n+i}b \}  $ for all $0 \leq i \leq n-1$. Thus, we have

\begin{equation}\label{eq(2)}
Q_{4n} = \langle a \rangle \bigcup\limits_{ i=  0}^{n-1} \langle a^ib \rangle.
\end{equation}


For $n \geq 2$, the \emph{semidihedral group} $SD_{8n}$ is a group of order $8n$ defined by the generators and relations

\[
SD_{8n} = \langle a, b : a^{4n} = e, \, b^2 = e, \, ba = a^{2n - 1}b \rangle.
\]
We have $$ ba^i = \left\{ \begin{array}{ll}
a^{4n -i}b & \mbox{if $i$ is even,}\\
a^{2n - i}b& \mbox{if $i$ is odd,}\end{array} \right.$$
so that every element of $SD_{8n} {\setminus} \langle a \rangle$ is of the form $a^ib$ for some $0 \leq i \leq 4n-1$. We denote the subgroups $P_i = \langle a^{2i}b \rangle = \{e, a^{2i}b\}$ and $ Q_j =  \langle a^{2j + 1}b \rangle = \{e, a^{2n}, a^{2j +1}b, a^{2n + 2j +1}b\} $. Then we have
\begin{equation}\label{Eq-3}
    SD_{8n} = \langle a \rangle \cup \left( \bigcup\limits_{ i=0}^{2n-1} P_i \right) \cup \left( \bigcup\limits_{ j=  0}^{n-1} Q_{j}\right),
\end{equation}
further, \[\ 
 Z(SD_{8n}) = \begin{cases}
    \{e,a^{2n}\},& \text{if $n$ is even} \\
    \{e,a^{2n},a^n,a^{3n}\},& \text{if $n$ is odd}.
    \end{cases}
\]

\section{$\mathcal{R}$-super graph of a graph}
 Let $\Gamma$ be a graph and let $\mathcal{R}$ be an equivalence relation on $V(\Gamma)$. Let $C_1, C_2, \ldots, C_k$ be the distinct $\mathcal{R}$-equivalence classes of $V(\Gamma)$ with $|C_i|=n_i$, for $1\leq i \leq k$. The $\mathcal{R}$-\emph{compressed} $\Gamma$ graph $\Im_{_{\Gamma^{\mathcal{R}}}}$ is a simple graph with $V(\Im_{_{\Gamma^{\mathcal{R}}}})=\{C_1, C_2, \ldots, C_k\}$ and two distinct vertices $C_i$ and $C_j$ are adjacent if there exist $x \in C_i$ and  $y \in C_j$ such that $x$ is adjacent to $y$ in $\Gamma$. The $\mathcal{R}$-\emph{super} $\Gamma$ graph $\Gamma^{\mathcal{R}}$ is a simple graph with vertex set $V(\Gamma)$ and two distinct vertices are join by an edge if either they are in same $\mathcal{R}$-equivalence class or there exist $x' \in [x]_{\mathcal{R}}$ and $y' \in [y]_{\mathcal{R}}$ such that $x' \sim y'$ in $\Gamma$. In this section, we discuss the characteristic polynomial of the  Sombor matrix of $\mathcal{R}$-super $\Gamma$ graph $\Gamma^{\mathcal{R}}$. The following results will be useful in the sequel.

\begin{theorem}{\cite[Proposition 3.1]{spectrumofsupergraphs}}
Consider a graph  $\Gamma$  and  $\mathcal{R}_1$  and $\mathcal{R}_2$  are two equivalence relations on $V(\Gamma)$.
If ${\mathcal{R}_1}\subseteq {\mathcal{R
}_2}$, then  $\Gamma^{\mathcal{R}_1}$ is a spanning subgraph of $\Gamma^{\mathcal{R}_2}$.
\end{theorem}

\begin{theorem} \label{isomorphism}{\cite[Theorem 3.2]{spectrumofsupergraphs}}
Consider a graph $\Gamma$ and let $\mathcal{R}$ is an equivalence relation on $V(\Gamma)$. Let $C_1, C_2, \ldots, C_k$ are the distinct  $\mathcal{R}$-equivalence classes of $V(\Gamma)$ with  $|C_i| = n_i$~~ for $1 \leq i \leq k.$ Then $\Gamma^\mathcal{R}$ is isomorphic to  $\Im_{_{\Gamma^{\mathcal{R}}}}[K_{n_1}, K_{n_2},\ldots, K_{n_k}]$.
\end{theorem}

\begin{theorem} \label{connectedness}{\cite[Theorem 3.3]{spectrumofsupergraphs}}
Consider a graph $\Gamma$ and let $\mathcal{R}$ is an equivalence relation on $V(\Gamma)$. Let $C_1, C_2, \ldots, C_k$ are the distinct  $\mathcal{R}$-equivalence classes of $V(\Gamma)$ and let $\Gamma_{i}$ are the induced subgraph of $\Gamma$ corresponding to the vertex set $C_i$, for  $1 \leq i \leq k$. If $\Gamma$ is connected, then $\Im_{_{\Gamma^{\mathcal{R}}}}$ is connected. Conversely, if $\Im_{_{\Gamma^{\mathcal{R}}}}$ is connected and all the $\Gamma_{i}$  are connected, then $\Gamma$ is connected.
\end{theorem}

In the following theorem, we obtain the characteristic polynomial of Sombor matrix  $S(\Gamma^\mathcal{R})$.
\begin{theorem} \label{Characteristic polynomial}
Let $\Gamma$ be a connected graph on $n$ vertices and $\mathcal{R}$ be an equivalence relation on  $V(\Gamma)$. Assume $C_1, C_2, \ldots, C_k$ be the distinct $\mathcal{R}$-equivalence classes of  $V(\Gamma)$  with  $|C_i| = n_i$  for $1 \leq i \leq k$. 
Then the characteristic polynomial of Sombor matrix  $S(\Gamma^\mathcal{R})$ is given by
    \[
    \chi(S(\Gamma^\mathcal{R}), x) = \chi(N, x)\prod_{i=1}^{k} \left(x + d_{i}\sqrt{2}\right)^{n_i - 1},
    \]
    where
    \[N =
    \displaystyle
    \begin{bmatrix}
    (n_1 - 1)d_1\sqrt{2} & n_{2}\sqrt{{d_1}^2 + {d_2}^2} & \cdots &  \cdots & n_{k}\sqrt{{d_1}^2 + {d_k}^2}  \\
     n_{1}\sqrt{{d_1}^2 + {d_2}^2} & (n_2 - 1)d_2\sqrt{2} & \cdots & \cdots & n_{k}\sqrt{{d_2}^2 + {d_k}^2} \\
    \vdots & \vdots & \ddots & &  \vdots \\
     \vdots & \vdots &  & \ddots & \vdots \\
    n_{1}\sqrt{{d_1}^2 + {d_k}^2} &  n_{2}\sqrt{{d_2}^2 + {d_k}^2}  & \cdots & \cdots &  (n_k - 1))d_k\sqrt{2}
    \end{bmatrix}.\]
  \begin{proof}
In view of Theorem \ref{connectedness}, $\Delta_{\Gamma^\mathcal{R}}$ is connected. 
Also, the subgraph induced by the vertices of the set $C_{i}$  is complete for each $1 \leq i \leq k$. Therefore, we have $\Gamma^\mathcal{R}$ is isomorphic to $\Delta_{\Gamma^\mathcal{R}}[K_{n_1}, K_{n_2}, \ldots, K_{n_k}]$ (see Theorem \ref{isomorphism}).
Let $\{v^{1}_{1},v^{1}_{2}, \ldots, v^{1}_{n_1},v^{2}_{2},\ldots, v^{2}_{n_2}, \ldots, v^{k}_{1},\ldots,v^{k}_{n_k}\}$ be a vertex labeling of the graph $\Gamma^{\mathcal{R}}$,
where $v^{i}_{j}\in C_i$, $1 \leq j \leq n_i$  and $1 \leq i \leq k$. For $1 \leq i \neq j \leq k$, we observed that the subgraph induced by the vertices belongs to the set $C_i \cup C_j$ in $\Gamma^{\mathcal{R}}$ is either complete or $x\nsim y$ for all $x \in C_i$ and $y \in C_j$. For $1 \leq i \leq k$, as a result,  we get $deg(v^i_j)$ is equal for all $1 \leq j \leq n_i$. Therefore, we suppose that $deg(v^i_j) = d_j$, where $1 \leq j \leq n_i$  and $1 \leq i \leq k$. With this labeling, the Sombor matrix of the graph $\Gamma^\mathcal{R}$ is
\begin{equation}
S(\Gamma^\mathcal{R}) =  
 \displaystyle \begin{bmatrix}
 
A_{1,1} &A_{1,2} &\cdots& \cdots& A_{1,k}\\
A_{2,1} & A_{2,2} &\cdots& \cdots& A_{2,k}\\
\vdots & \vdots &   \ddots & & \vdots & \\

\vdots & \vdots & &  \ddots & \vdots &  \\ 
A_{k,1} &A_{k,2} &\cdots& \cdots & A_{k,k}\\
\end{bmatrix},
\end{equation}
where, 
\[ 
\begin{matrix}
  A_{i,i}= d_{i}\sqrt{2}(J_{{n_i} \times{n_i}}-I_{{n_i} \times{n_i}}),~~  
   \text{for~~ $1 \leq i \leq k$}\\
\\
 and\;  A_{i,j}= \begin{dcases}
      (\sqrt{(d_{i})^2 + (d_{j})^2}) J_{{n_i} \times{n_j}} ; \text{if} ~~v^{i}_{n_i} \sim v^{j}_{n_j},\\
          0 ~~~~ ;  ~~~~~~~~~  otherwise   
 , \end{dcases}
\end{matrix}
\]
Note that for each $i\in\{1,2,\ldots,k\}$, the vertices $v^{i}_{n_i}$ in $\Gamma$ form a clique $K_{n_i}$ and each vertex of $K_{n_i}$ share the same neighborhood. Therefore, by Lemma \ref{multiplicity}, $-d_{1}\sqrt{2}, -d_{2}\sqrt{2}, \ldots , -d_{k}\sqrt{2}$ are the eigenvalues of the Sombor matrix with multiplicities $ (n_1-1), (n_2-1), \ldots, (n_k-1)$,  respectively.     
In view of Theorem \ref{equitable matrix}, the remaining eigenvalues of the Sombor matrix of $\Gamma^{R}$ are the roots of the equitable quotient matrix given in $(2)$.
\end{proof}
\end{theorem}

The following corollary is a consequence of Theorem \ref{Characteristic polynomial}.

\begin{corollary} \label{Generalize join}
Let $\Gamma$ be the generalized join graph $K_{1,k-1}[K_{n_1}, K_{n_2}, \ldots, K_{n_k}]$, where $n= n_1 + n_2 + \cdots + n_k$. Then the characteristic polynomial of $S(\Gamma)$ is given by
\begin{align}
    \chi(S(\Gamma), x) = & \prod_{i=1}^{k} \left(x + \sqrt{2} d_i\right)^{n_i - 1} 
      \left( \prod_{i=1}^{k} \{x - (n_{i}-1)\sqrt{2} d_i\} - n_{1}n_{2}({d_1}^2 + {d_2}^2) \prod_{i=3}^{k} \{x - (n_{i}-1)\sqrt{2} d_i\} \right. \nonumber \\
    & \left. - \cdots - n_{1}n_{k} ({d_1}^2 + {d_k}^2) \prod_{i=2}^{k-1} \{x - (n_{i}-1)\sqrt{2} d_i\} \right),
\end{align}
where $d_1 = n-1$ and $d_i = n_1 + n_i -1$ with $2 \leq i \leq k$.
\begin{proof}
Suppose that $\{v^{1}_{1},v^{1}_{2}, \ldots, v^{1}_{n_1},v^{2}_{1},\ldots, v^{2}_{n_2}, \ldots, v^{k}_{1},\ldots,v^{k}_{n_k} \}$ be the vertex labeling of $\Gamma$, where $v^{i}_{j}$ are the vertices of $K_{n_i}$ with $1 \leq j \leq n_i$. Note that N$[v^1_j] = V(\Gamma)$ for all $1 \leq j \leq n_1$ and N$[v^i_j] = V(K_{n_i}) \cup V(K_{n_1})$ for all $2 \leq i \leq k$. Therefore, we have $deg(v^{1}_{j})= n- 1= d_1$ for all $1 \leq j \leq n_1$ and $deg(v^{i}_{j})= n_1 + n_i - 1 = d_i$ for all $2 \leq i \leq k$ and $1 \leq j \leq n_i$. In view of Theorem \ref{Characteristic polynomial}, we have

\[
    \chi(S(\Gamma), x) = \chi(N, x)\prod_{i=1}^{k} \left(x +  \sqrt{2} d_i \right)^{n_i - 1},
    \]
    where
    \[
    N =
    \displaystyle
    \begin{bmatrix}
    (n_1 - 1)\sqrt{2}d_1 & n_{2}\sqrt{{d_1}^2 + {d_2}^2} & \cdots &  \cdots & n_{k}\sqrt{{d_1}^2 + {d_k}^2}  \\
     n_{1}\sqrt{{d_1}^2 + {d_2}^2} & (n_2 - 1)d_2\sqrt{2} & \cdots & \cdots & 0 \\
    \vdots & \vdots & \ddots & &  \vdots \\
     \vdots & \vdots &  & \ddots & \vdots \\
    n_{1}\sqrt{{d_1}^2 + {d_k}^2} &  0  & \cdots & \cdots &  (n_k - 1))d_k\sqrt{2}
    \end{bmatrix}.
    \]
The characteristic polynomial of matrix $N$ is given by  
\[
 \chi(N, x) =
    \displaystyle
    \begin{vmatrix}
    x-(n_1 - 1)d_1\sqrt{2} & n_{2}\sqrt{{d_1}^2 + {d_2}^2} & \cdots &  \cdots & n_{k}\sqrt{{d_1}^2 + {d_k}^2}  \\
     n_{1}\sqrt{{d_1}^2 + {d_2}^2} & x-(n_2 - 1)d_2\sqrt{2} & \cdots & \cdots & 0 \\
    \vdots & \vdots & \ddots & &  \vdots \\
     \vdots & \vdots &  & \ddots & \vdots \\
    n_{1}\sqrt{{d_1}^2 + {d_k}^2} &  0  & \cdots & \cdots & x-(n_k - 1))d_k\sqrt{2}
    \end{vmatrix}.
    \]
Expanding the above matrix along the first row, we get the characteristic polynomial
\begin{align*}
    \chi(S(\Gamma), x) = & \prod_{i=1}^{k} \left(x + d_i\sqrt{2} \right)^{n_i - 1} 
      \left( \prod_{i=1}^{k} \{x - (n_{i}-1)d_i\sqrt{2} \} - n_{1}n_{2}({d_1}^2 + {d_2}^2) \prod_{i=3}^{k} \{x - (n_{i}-1)d_i\sqrt{2} \} \right. \nonumber \\
    & \left. - \cdots - n_{1}n_{k} ({d_1}^2 + {d_k}^2) \prod_{i=2}^{k-1} \{x - (n_{i}-1)d_i\sqrt{2} \} \right).
\end{align*}
\end{proof} 
\end{corollary}

In view of Corollary \ref{Generalize join}, we have the following result.

\begin{corollary} \label{generalize join Spectrum}
Let $\Gamma$ be the generalized join graph $K_{1,k-1}[K_l, K_{m_2}, K_{m_3},  \ldots, K_{m_{k}}]$ with $m_i = m$ for all $2 \leq i \leq k$ and $n=l+(k-1)m$. Then 
\[
\sigma(S(\Gamma)) = \displaystyle \begin{pmatrix}
-(n-1)\sqrt{2} & -(l+m-1)\sqrt{2} & (m-1)(l+m-1)\sqrt{2} & y_1 & y_2 \\
 l-1 & mk-k-m+1 & k-2 & 1 & 1 \\
 \end{pmatrix}, 
\]
where $y_1= \frac{1}{2}\left[(l-1)d_1\sqrt{2}+(m-1)d_2\sqrt{2} +\sqrt{[(l-1)d_1\sqrt{2}-(m-1)d_2\sqrt{2}]^2+4lm(k-1)({d_1}^2 + {d_2}^2)}\right]$ and $y_2= \frac{1}{2}\left[(l-1)d_1\sqrt{2}+(m-1)d_2\sqrt{2} -\sqrt{[(l-1)d_1\sqrt{2}-(m-1)d_2\sqrt{2}]^2+4lm(k-1)({d_1}^2 + {d_2}^2)}\right],$ with $d_1 = n-1$ and $d_2= m+l -1$.
\end{corollary}

\begin{proof} Suppose that $\{v^{1}_{1},v^{1}_{2}, \ldots, v^{1}_{l},v^{2}_{1},\ldots, v^{2}_{m_2}, \ldots, v^{k}_{1},\ldots,v^{k}_{m_k} \}$ be the vertex labeling of $\Gamma$, where $v^{i}_{j}$ are the vertices of $K_{m_i}$ with $1 \leq j \leq m_i$ and $2 \leq i \leq k$; for $1 \leq j \leq l$, $v^1_j \in V(K_l)$.  By Corollary \ref{Generalize join}, the characteristic polynomial of Sombor matrix is given by $\chi(S(\Gamma), x)=
(x + d_1\sqrt{2})^{l-1}(x + d_2\sqrt{2})^{(m-1)(k-1)} ~\times$\\
$\left[(x-(l-1)d_1\sqrt{2})(x-(m-1)d_2\sqrt{2})^{k-1} - lm(k-1)({d_1}^2 + {d_2}^2)(x-(m-1)d_2\sqrt{2})^{k-2}\right]$
\begin{align*}
& =(x + d_1\sqrt{2})^{l - 1}(x + d_2\sqrt{2})^{(m-1)(k-1)}(x-(m-1)d_2\sqrt{2})^{k-2}\left[(x-(l-1)d_1\sqrt{2})(x-(m-1)d_2\sqrt{2})- lm(k-1)({d_1}^2 + {d_2}^2)\right]\\
&=(x + d_1\sqrt{2})^{l - 1}(x + d_2\sqrt{2})^{(m-1)(k-1)}(x-(m-1)d_2\sqrt{2})^{k-2} (x-y_1)(x-y_2),
\end{align*}
where $y_1= \frac{1}{2}\left[(l-1)d_1\sqrt{2}+(m-1)d_2\sqrt{2} +\sqrt{[(l-1)d_1\sqrt{2}-(m-1)d_2\sqrt{2}]^2+4lm(k-1)({d_1}^2 + {d_2}^2)}\right]$ and \\
$y_2= \frac{1}{2}\left[(l-1)d_1\sqrt{2}+(m-1)d_2\sqrt{2} -\sqrt{[(l-1)d_1\sqrt{2}-(m-1)d_2\sqrt{2}]^2+4lm(k-1)({d_1}^2 + {d_2}^2)}\right].$
\end{proof}

The notation $G$ for a group and $e$ for its identity element are fixed throughout the paper.

\section{Super commuting graph}
The study of commuting graphs associated with groups is an important area of research in algebra and graph theory. It was introduced by Brauer and Fowler \cite{a.1955-Brauer-group} because of commuting graph provides valuable insights into the structural properties of a group by examining the relationships between its elements. Later, many authors studied the graph theoretic and algebraic properties of commuting graph $\Delta(G)$ (see \cite{a.Araujo2015symmetricinverse,a.Dolvazn2017,a.Shitov2018}). In this section, we discuss the Sombor spectrum of $\Delta(G), \Delta^o(G)$ and $\Delta^c(G)$, where $G \in \{D_{2n}, Q_{4n}, SD_{8n}\}$ into three subsections.

\subsection{Commuting Graph}
In this section, we discuss the Sombor spectrum of the commuting graph $\Delta(G)$, for the group $G= D_{2n}, Q_{4n}$ and $SD_{8n}.$

\begin{corollary} \label{commuting-D_2n}
Let $\Delta(D_{2n})$ be the commuting graph of the dihedral group $D_{2n}$.  
\begin{itemize}
    \item [(i)] If $n$ is odd, then the Sombor spectrum of $\Delta(D_{2n})$ is 
    \[\displaystyle \begin{pmatrix}
-(n-1)\sqrt{2} &  0    \\
 n-2 & n-1   \\
\end{pmatrix},
\]
\noindent and the remaining eigenvalues are the roots of the following polynomial\\
$x^3 +   \sqrt{2}(3n -n^2 - 2)x^2 + (15n^2-9n^3 -10n + 2)x + \sqrt{2}(4n^5  -16n^4 + 22n^3 -14n^2 + 4n)$.
\vspace{0.2cm}
 \item[(ii)] If $n$ is even , then the Sombor spectrum of $\Delta(D_{2n})$ is
     \[\displaystyle \begin{pmatrix}
-(2n-1)\sqrt{2} & -(n-1)\sqrt{2} & -3\sqrt{2}\\
1 &  n-3 & \frac{n}{2}\\
\end{pmatrix},\]
\noindent and the remaining eigenvalues are the roots of the following polynomial\\ 
$\Big[ x - (2n-1)\sqrt{2} \Big] \Big[ x - (n-1)(n-3)\sqrt{2} \Big] \Big[ x - 3\sqrt{2} \Big]^{\frac{n}{2}}
 - 2(n-2)(5n^2 -6n +2) \Big[ x -3\sqrt{2}\Big]^{\frac{n}{2}}- 2n (7n^2 -4n +10)\times$\\
 \\
$ \Big[ x - (n-3)(n-1)\sqrt{2}\Big]\Big[ x - 3\sqrt{2} \Big]^{\frac{n}{2}-1}$.
\end{itemize}
\end{corollary}

\begin{proof}
(i) For odd $n$, we have
$\Delta(D_{2n})= K_{1,2}[K_1,K_{n-1},\overline{K}_n]$ (cf. \cite{arunkumar2024graphs}). Now consider $A_1 = \{e\}, A_2 = \langle a \rangle \setminus \{e\}, C_i = \{a^ib\}$ with $0 \leq i \leq n-1$. Note that $|A_1| = 1$, $|A_2| = n-1$ and $|C_i| = 1$ for all $0 \leq i \leq n-1$. In view of Corollary \ref{Generalize join}, we get\\

$\chi(S(\Delta(D_{2n})), x) = \Big[x + (n-1)\sqrt{2}\Big]^{n - 2} \times  \Bigg\{x^{n+1}\Big[x-(n-1)(n-2)\sqrt{2}\Big]- (n-1)\Big[(2n-1)^2 + (n-1)^2\Big]x^n$

$-\Big[(2n-1)^2 + 1\Big]\Big[x - \sqrt{2}(n-2)(n-1)\Big]x^{n-1} - \cdots - \Big[(2n-1)^2 + 1\Big]\Big[x - \sqrt{2}(n-2)(n-1)\Big]x^{n-1} \Bigg\}$

\vspace{0.3cm}

\noindent$=\Big[x + (n-1)\sqrt{2}\Big]^{n - 2} \times \Bigg\{x^{n+1}\Big[x-(n-1)(n-2)\sqrt{2}\Big]- (n-1)\Big[(2n-1)^2 + (n-1)^2\Big]x^n$ 
$-n\Big[(2n-1)^2 + 1\Big]\times$\\ $\Big[x - (n-2)(n-1)\sqrt{2}\Big]x^{n-1} \Bigg\}$

\vspace{0.3cm}

\noindent$= \Bigg\{ x^2\Big[x-(n-1)(n-2)\sqrt{2}\Big]- (n-1)\Big[(2n-1)^2 + (n-1)^2]x-n\Big[(2n-1)^2 + 1\Big]\Big[x - \sqrt{2}(n-2)(n-1)\Big] \Bigg\} \times$

$\Big[x + (n-1)\sqrt{2}\Big]^{n - 2} x^{n-1}$
\vspace{0.3cm}

\noindent$= \Big[x + (n-1)\sqrt{2}\Big]^{n - 2} x^{n-1} \Bigg\{x^3 +   \sqrt{2}(3n -n^2 - 2)x^2 + (15n^2-9n^3 -10n + 2)x + \sqrt{2}(4n^5  -16n^4 \textcolor{red}{+}22n^3 -14n^2 + 4n)\Bigg\}
$.\\
Thus, the Sombor spectrum of $\Delta(D_{2n})$  is 
    \[\displaystyle \begin{pmatrix}
-(n-1)\sqrt{2} & 0 \\
 n-2 & n-1 \\
\end{pmatrix},
\]
and the remaining eigenvalues are the roots of the following polynomial\\ 
$x^3 +   \sqrt{2}(3n -n^2 - 2)x^2 + (15n^2-9n^3 -10n + 2)x + \sqrt{2}(4n^5  -16n^4 +22n^3 -14n^2 + 4n)$. 

\vspace{0.3cm}
\noindent(ii) As $n$ is even, the $\Delta(D_{2n})= K_{1,\frac{n}{2}+1}[K_2,K_{n-2}, \underbrace{K_2, K_2, \ldots, K_2}_{\frac{n}{2}-{\rm times}}]$ (cf. \cite{arunkumar2024graphs}). Now consider $A_1 = \{e, a^n\}, A_2 = \langle a \rangle \setminus \{e, a^n\}, C_i = \{a^ib, a^{i + \frac{n}{2}}b\}$ with $1 \leq i \leq \frac{n}{2}$. Note that $|A_1| = 2$, $|A_2| = n-2$ and $|C_i| = 2$ for all $1 \leq i \leq \frac{n}{2}$. In view of Corollary \ref{Generalize join}, we get\\ $\chi(S(\Delta(D_{2n})), x) =$
$ \Big[x + (2n-1)\sqrt{2}\Big] \Big[x + (n-1)\sqrt{2}\Big]^{n - 3} \Big[x + 3\sqrt{2}\Big]^{\frac{n}{2}} \times \Bigg\{\Big[x - (2n-1)\sqrt{2}\Big] \Big[x - (n-1)(n-3)\sqrt{2}\Big] \Big[x - 3\sqrt{2}\Big]^{\frac{n}{2}}$ $- 2(n-2)\Big[ (2n-1)^2 + (n-1)^2\Big] \Big[ x -3\sqrt{2}\Big]^{\frac{n}{2}}
- 2n \Big[ (2n-1)^2 + 9\Big]  \Big[ x - (n-3)(n-1)\sqrt{2}\Big] \Big[ x - 3\sqrt{2}\Big]^{\frac{n}{2}-1} \Bigg\}$.

Thus, the Sombor spectrum is
     \[\displaystyle \begin{pmatrix}
-(2n-1)\sqrt{2} & -(n-1)\sqrt{2} & -3\sqrt{2} \\
1 &  n-3 & \frac{n}{2}\\
\end{pmatrix},\]
and the remaining eigenvalues are the roots of the following polynomial\\
$\Big[x - (2n-1)\sqrt{2}\Big] \Big[x - (n-1)(n-3)\sqrt{2}\Big] \Big[x - 3\sqrt{2}\Big]^{\frac{n}{2}}$ $- 2(n-2)(5n^2 -6n + 2) \Big[ x -3\sqrt{2}\Big]^{\frac{n}{2}}
- 4n ( 2n^2 -2n +5) \Big[ x - (n-3)(n-1)\sqrt{2}\Big] \Big[ x - 3\sqrt{2}\Big]^{\frac{n}{2}-1}$.
\end{proof}

\begin{corollary} \label{commuting-Q_4n}
Let $\Delta(Q_{4n})$ be the commuting graph of the generalized quaternion group. Then the Sombor spectrum of $\Delta(Q_{4n})$ is given below 

    \[\displaystyle \begin{pmatrix}
-(4n-1)\sqrt{2} & -(2n-1)\sqrt{2} & -3\sqrt{2} \\
 1 & 2n-3 & n\\
\end{pmatrix},
\]
and the  remaining eigenvalues are the roots of the following polynomial\\ 
$\Big[x-(4n-1)\sqrt{2} \Big]\Big[ x - (2n-1)(2n-3)\sqrt{2} \Big] \Big[ x -3\sqrt{2}\Big]^n  
 - 8(n-1)(10n^2 -6n +1) \Big[ x -3\sqrt{2}\Big]^n
- 8n(8n^2 -4n +5)\times$\\
\\
$\Big[ x - (2n-1)(2n-3)\sqrt{2} \Big]\Big[ x -3\sqrt{2}\Big]^{n-1}$.
\end{corollary} 

\begin{proof}The commuting graph of the generalized quaternion group $Q_{4n}$ (see \cite{arunkumar2024graphs}) is
 \[
\Delta(Q_{4n})=
    K_{1,n+1}[K_2,K_{2n-2},\underbrace{K_{2},\ldots,K_{2}}_{n-{\rm times}}].
\]
 By Corollary \ref{Generalize join}, we get\\
$
\chi(S(\Delta(Q_{4n})), x) =  \Big[x + (2n-1)\sqrt{2}\Big]^{2n-3} \Big[x + (4n-1) \sqrt{2} \Big] \Big[x + 3\sqrt{2}\Big]^n \times
  \Bigg\{ \Big[x-(4n-1)\sqrt{2} \Big]\Big[ x - (2n-1)(2n-3)\sqrt{2} \Big] \Big[ x -3\sqrt{2}\Big]^n 
 -2(2n-2) \Big[ (4n-1)^2 + (2n-1)^2 \Big]\Big[ x -3\sqrt{2}\Big]^n 
 -4n \Big[ (4n-1)^2 + 3^2 \Big] \Big[ x - (2n-1)(2n-3)\sqrt{2} \Big] \Big[ x -3\sqrt{2}\Big]^{n-1} \Bigg\} $.\\
 It follows that $-(4n-1)\sqrt{2}$, $-(2n-1)\sqrt{2}$ and $-3\sqrt{2}$ are the Sombor eigenvalues of ${\Delta}(Q_{4n})$ with its multiplicities $1$, $2n-3$ and $n$, respectively and  the remaining eigenvalues are the roots of the polynomial \\
$\Big[x-(4n-1)\sqrt{2} \Big]\Big[ x - (2n-1)(2n-3)\sqrt{2} \Big] \Big[ x -3\sqrt{2}\Big]^n  
 - 8(n-1)(10n^2 -6n +1) \Big[ x -3\sqrt{2}\Big]^n
- 8n(8n^2 -4n +5)\times$\\$\Big[ x - (2n-1)(2n-3)\sqrt{2} \Big]\Big[ x -3\sqrt{2}\Big]^{n-1}$.
\end{proof}

\begin{corollary} \label{commuting-SD_8n}
Let $G$ be semidihedral group $SD_{8n}$.
\begin{itemize}
    \item [(i)] For odd $n$, the Sombor spectrum of $\Delta(SD_{8n})$ is
    \[\displaystyle \begin{pmatrix}
-(8n-1)\sqrt{2} & -(4n-1)\sqrt{2} & -7\sqrt{2} & 21\sqrt{2}\\
 3 & 4n-5 & 3n & n -1\\
\end{pmatrix},
\]
\noindent and the remaining eigenvalues are the roots of the following polynomial 

  \noindent $ \Big[x-3(8n-1)\sqrt{2} \Big]\Big[ x - (4n-5)(4n-1)\sqrt{2} \Big] \Big[ x - 21\sqrt{2}\Big]
 -32(n-1) (40n^2-12n + 1) \Big[ x - 21\sqrt{2}\Big]$\\
\\ 
 $-32n(32n^2-8n +25) \Big[ x - (4n-1)(4n-5)\sqrt{2} \Big]$.

\vspace{0.2cm}
 \item[(ii)] For even $n$, the Sombor spectrum of $\Delta(SD_{8n})$ is 
     \[\displaystyle \begin{pmatrix}
-(8n-1)\sqrt{2} & -(4n-1)\sqrt{2} & -3\sqrt{2} & 3\sqrt{2}\\
1 &  4n-3 & 2n & 2n-1 \\
\end{pmatrix},\]
\noindent and remaining eigenvalues are the roots of the following polynomial\\ 
$ \Big[x-(8n-1)\sqrt{2} \Big] \Big[ x - (4n-1)(4n-3)\sqrt{2} \Big] \Big[ x -3\sqrt{2}\Big]
 - 8(2n-1)  (40n^2-12n+1 ) \Big[ x -3\sqrt{2}\Big]$\\
\\
$- 16n (32n^2-12n + 1) \Big[ x - (4n-1)(4n-3)\sqrt{2} \Big]$.  

\end{itemize}
\end{corollary}
\begin{proof}
 The commuting graph of the semidihedral group $SD_{8n}$ is given below (cf. \cite{dalal2021semidihedral})
 \[
\Delta(SD_{8n})= \begin{dcases}
    K_{1,n+1}[K_4,K_{4n-4},\underbrace{K_{4},\ldots,K_{4}}_{n-{\rm times}}], & \text{if $n$ is odd},\\
    K_{1,2n+1}[K_2,K_{4n-2},\underbrace{K_{2},\ldots,K_{2}}_{2n-{\rm times}}], & \text{if $n$ is even}. \\
\end{dcases}
\]

First, we assume that $n$ is odd. 
In view of Corollary \ref{Generalize join}, one can obtain the Sombor spectrum by using  the same arguments used in the Corollaries \ref{commuting-D_2n} and \ref{commuting-Q_4n}. 
 For even $n$,  we have $Com(SD_{8n}) \cong Com(Q_{4m})$ and the Sombor spectrum of  $Com(Q_{4m})$ has been discussed in Corollary \ref{commuting-Q_4n}.
\end{proof}

\subsection{Order Supercommuting Graph}
In this subsection, we discuss the Sombor spectrum of the order supercommuting graph $\Delta^o(G)$, for $G= D_{2n}, Q_{4n}$ and $SD_{8n}$. Dalal \emph{et al.} \cite{spectrumofsupergraphs}, gave the structure of the following graphs:

\[
\Delta^o(D_{2n})= 
\begin{dcases}
    K_{1,2}[K_1,K_{n-1},K_n], \text{ if $n$ is odd}   \\
   K_{2n},   \text{ if $n$ is even}  \\
\end{dcases}
\]
 and
 \[
\Delta^o(Q_{4n})= 
\begin{dcases}
    K_{1,2}[K_2,K_{2n-2},K_{2n}], \text{ if $n$ is odd}   \\
   K_{4n},   \text{ if $n$ is even}.   \\
\end{dcases}
\]

By Corollary \ref{Generalize join}, one can verify the proof of the following Corollaries.

\begin{corollary} \label{OrderSC-D_2n}
Let $\Delta^o(D_{2n})$ be an order supercommuting graph of the dihedral group.  
\begin{itemize}
    \item [(i)] If $n$ is odd, then the Sombor spectrum of $\Delta^o(D_{2n})$ is
    \[\displaystyle \begin{pmatrix}
-(n-1)\sqrt{2} & -n\sqrt{2}\\
 n-2 & n-1 \\
\end{pmatrix},
\]
\noindent and the remaining eigenvalues are the roots of the following polynomial\\ $x \Big[ x - (n-1)(n-2)\sqrt{2}\Big]\Big[ x - n(n-1)\sqrt{2}\Big] - (n-1)(5n^2 -6n +2) \Big[x - n(n-1)\sqrt{2}\Big] - n(5n^2 -4n +1)\times$\\
 \\
 $ \Big[ x - (n-1)(n-2)\sqrt{2}\Big].$ 

\vspace{0.2cm}

 \item[(ii)] If $n$ is even, then the Sombor spectrum of $\Delta^o(D_{2n})$ is
     \[\displaystyle \begin{pmatrix}
-(2n-1)\sqrt{2} & (2n-1)^2\sqrt{2} \\
 2n-1 & 1 \\
\end{pmatrix}.\]
\end{itemize}
\end{corollary}


\begin{corollary} \label{OrderSC-Q_4n} 
Let $\Delta^o(Q_{4n})$ be the order supercommuting graph of the generalized quaternion group.  
\begin{itemize}
\item [(i)] If $n$ is odd, then the Sombor spectrum of $\Delta^o(Q_{4n})$ is 
    \[\displaystyle \begin{pmatrix}
-(4n-1)\sqrt{2} & -(2n-1)\sqrt{2} & -(2n+1)\sqrt{2}\\
 1 & 2n-3 & 2n-1 \\
\end{pmatrix},
\] 
\noindent and the remaining eigenvalues are the roots of the following polynomial\\ 
$\Big[x-(4n-1)\sqrt{2}\Big]\Big[ x - (2n-1)(2n-3)\sqrt{2}\Big] \Big[x - (2n-1)(2n+1)\sqrt{2}\Big]- 8(n-1)( 10n^2-6n + 1)\times$\\
\\
$\Big[x - (2n-1)(2n+1)\sqrt{2}\Big] 
-8n(10n^2 -2n + 1)\Big[x - (2n-1)(2n-3)\sqrt{2}\Big]$.

\vspace{0.2cm}
\item[(ii)] If $n$ is even, then the Sombor spectrum of $\Delta^o(Q_{4n})$ is
     \[\displaystyle \begin{pmatrix}
-(4n-1)\sqrt{2} & (4n-1)^2\sqrt{2} \\
 4n-1 & 1 \\
\end{pmatrix}.\]
\end{itemize}
\end{corollary}





\begin{corollary} \label{OrderSC_SD_8n}
Let $\Delta^o(SD_{8n})$ be an order super commuting graph of the semidihedral group $SD_{8n}$. Then the Sombor spectrum  of order super commuting graph $\Delta^o(SD_{8n})$ is  
\[\displaystyle \begin{pmatrix}
-(8n-1)\sqrt{2} & (8n-1)^2\sqrt{2} \\
 8n-1 & 1 \\
\end{pmatrix}.
\]
\end{corollary}

\begin{proof}
In the semidihedral group $SD_{8n}$, we consider $C_1 = \{a^{2i}b: 0 \leq i \leq 2n-1\} \cup \{ a^{2n}\} $, $ C_2 = \{a^{2i+1}b: 0 \leq i \leq 2n-1\} \cup \{ a^{n}, a^{3n}\} $ and $C_3= \left< a\right> - \{ e, a^{n}, a^{2n}, a^{3n} \}$. Further, note that the order of each vertex belongs to the set $C_1$ is two and $o(x) = 4$ for all $x \in C_2$. Therefore, the subgraphs induced by the vertices belong to the sets $C_1, C_2$ and $C_3$ are complete graphs, respectively. As $a^{2n} \in Z(SD_{8n})$ (cf. Eq. (1)), gives $a^{2n}$ adjacent to every vertex of  $C_2$ and $C_3$ in 
$OSCom(SD_{8n})$. This implies that $N[C_1]= SD_{8n}$. We noticed that $N[a^n] = SD_{8n}$. Therefore, we have $N[C_2]= N[C_3]= SD_{8n}$. By Theorem \ref{multiplicity}, $-(8n-1)\sqrt{2}$ is an eigenvalue of multiplicity $8n-1$. Since the sum of the eigenvalues of Sombor matrix of  $OSCom(SD_{8n})$ is equal to its trace, which is zero, the remaining eigenvalue is equal to $(8n-1)^2\sqrt{2}$.
\end{proof}

\subsection{Conjugacy Supercommuting Graph}
In this subsection, we discuss the Sombor spectrum of $\Delta^c(G)$, for $G= D_{2n}, Q_{4n}$ and $SD_{8n}$. Dalal \emph{et al.} \cite{spectrumofsupergraphs}, gave a representation of the following graphs: 

\[
\Delta^c(Q_{4n})= 
\begin{dcases}
    K_{1,2}[K_2,K_{2n-2},K_{2n}],& \text{ if $n$ is odd}   \\
    K_{1,3}[K_2,K_{2n-2},K_{n},K_{n}],&   \text{ if $n$ is even}   \\
\end{dcases}
\]
and 
\[
\Delta^c(D_{2n}) \cong
\begin{dcases}
    K_{1,2}[K_1,K_{n-1},K_n],& \text{ if $n$ is odd}   \\
    K_{1,2}[K_2,K_{n-2},K_n],&  \text{ if $n$ is even and $\frac{n}{2}$ is odd }   \\
      K_{1,3}[K_2,K_{n-2},K_{\frac{n}{2}}, K_{\frac{n}{2}}],&  \text{ if $n$ is even and $\frac{n}{2}$ is even}.   \\
  \end{dcases}
\]

Also, the structure of conjugacy super computing graph of semidihedral group $SD_{8n}$ (see \cite{das2024super}) is 
\[
\Delta^c(SD_{8n})= 
\begin{dcases}
    K_{1,2}[K_4,K_{4n-4},K_{4n}],& \text{if $n$ is odd}   \\
    K_{1,3}[K_2,K_{4n-2},K_{2n},K_{2n}],& \text{if $n$ is even}.   \\
\end{dcases}
\]

By Corollary \ref{Generalize join}, the proof of the following Corollaries are straightforward.

\begin{corollary} \label{ConjugacySC-SD_8n}
In $\Delta^c(SD_{8n})$, 
\begin{itemize}
\item [(i)] for odd $n$, the Sombor spectrum of conjugacy super commuting graph $\Delta^c(SD_{8n})$ is 
    \[\displaystyle \begin{pmatrix}
-(8n-1)\sqrt{2} & -(4n-1)\sqrt{2} & -(4n+3)\sqrt{2}\\
 3 & 4n-5 & 4n-1\\
\end{pmatrix}
,\]
\noindent and the remaining eigenvalues are the roots of the following  polynomial\\
$\Big[x - 3(8n-1)\sqrt{2}\Big] \Big[ x - (4n-1)(4n-5)\sqrt{2}\Big] \Big[x - (4n-1)(4n+3)\sqrt{2}\Big]- 32(n-1) (40n^2-12n +1) \times$\\ 
\\
\noindent$\Big[x - (4n-1)(4n+3)\sqrt{2}\Big]- 32n (40n^2 + 4n + 5) \Big[x - (4n-5)(4n-1)\sqrt{2}\Big]$. 
\vspace{0.2cm}
\item[(ii)] for even $n$, the Sombor spectrum of conjugacy super commuting graph $\Delta^c(SD_{8n})$ is 
     \[\displaystyle \begin{pmatrix}
-(2n+1)\sqrt{2} &  -(4n-1)\sqrt{2}  &  -(8n-1)\sqrt{2}\\
 4n-2 & 4n-3  & 1 \\
\end{pmatrix}
,\]
and the remaining eigenvalues are the roots of the following polynomial\\
$\Big[x - (8n-1)\sqrt{2}\Big] \Big[x - (4n-1)(4n-3)\sqrt{2}\Big] \Big[x - (2n-1)(2n+1)\sqrt{2}\Big]^2
 - 8(2n-1)(40n^2 -12n +1) \times$ \\ 
\\ 
 $\Big[x - (2n-1)(2n+1)\sqrt{2}\Big]^2
 - 16n (34n^2 - 6n +1 ) \Big[x - (4n-3)(4n-1)\sqrt{2}\Big] \Big[x - (2n-1)(2n+1)\sqrt{2}\Big]$. 

 \end{itemize}
\end{corollary}

\begin{corollary} \label{ConjugacySC-Q_4n}
Let $\Delta^c(Q_{4n})$ be the conjugacy super commuting graph of the generalized quaternion group.  
\begin{itemize}
    \item [(i)] If $n$ is odd, then the Sombor spectrum of conjugacy super commuting graph $\Delta^c(Q_{4n})$ is 
    \[\displaystyle \begin{pmatrix}
-(4n-1)\sqrt{2} & -(2n-1)\sqrt{2} & -(2n+1)\sqrt{2} \\
 1 & 2n-3 & 2n-1  \\
\end{pmatrix},
\]
\noindent and the remaining eigenvalues are the roots of the following polynomial\\
\noindent$\Big[x-(4n-1)\sqrt{2}\Big]\Big[x - (2n-1)(2n-3)\sqrt{2}\Big] \Big[x - (2n-1)(2n+1)\sqrt{2} \Big]
 - 8(n-1) (10n^2 - 6n +1) \times$\\
 \\
 $ \Big[x - (2n-1)(2n+1)\sqrt{2}\Big]
- 8n (10n^2 -2n + 1) \Big[x - (2n-1)(2n-3)\sqrt{2}\Big].$
\vspace{0.2cm}
 \item[(ii)] If $n$ is even, then the Sombor spectrum of conjugacy super commuting graph $\Delta^c(Q_{4n})$ is 
     \[\displaystyle \begin{pmatrix}
-(4n-1)\sqrt{2} & -(2n-1)\sqrt{2} & -(n+1)\sqrt{2} \\
 1 & 2n-3 & 2n-2  \\
\end{pmatrix},\]
\noindent and the remaining eigenvalues are the roots of the following polynomial\\
$\Big[x - (4n-1)\sqrt{2}\Big] \Big[x - (n-1)(n+1)\sqrt{2}\Big]^2 \Big[x - (2n-1)(2n-3)\sqrt{2}\Big] 
 - 8(n-1) (10n^2 - 6n +1) \Big[x - (n-1)(n+1)\sqrt{2}\Big]^2 $\\
 \\
 $- 4n (17n^2 -6n + 2) \Big[x - (2n-3)(2n-1)\sqrt{2}\Big]\Big[x - (n-1)(n+1)\sqrt{2}\Big]$.
\end{itemize}
 \end{corollary}

\begin{corollary} \label{ConjugacySC-D_2n}
Let $\Delta^c(D_{2n})$ be the conjugacy super commuting graph of the dihedral group.
\begin{itemize}
    \item [(i)] If $n$ is odd, then the Sombor spectrum of conjugacy super commuting graph $\Delta^c(D_{2n})$ is 
    \[\displaystyle \begin{pmatrix}
-(n-1)\sqrt{2} & -n\sqrt{2} \\
 n-2 & n-1 \\
\end{pmatrix},
\]
\noindent and remaining eigenvalues are the roots of the following polynomial\\  
$x\Big[ x - (n-1)(n-2)\sqrt{2}\Big] \Big[x-n(n-1)\sqrt{2}\Big]
     - (n-1)(5n^2 -6n + 2)\Big[x - n(n-1)\sqrt{2}\Big] 
 - n(5n^2 -4n + 1)\times$\\
 \\
 $\Big[x - (n-1)(n-2)\sqrt{2}\Big]$.
\vspace{0.2cm}
\item[(ii)] If $n$ is even and $\frac{n}{2}$ is odd, then the Sombor spectrum of conjugacy super commuting graph $\Delta^c(D_{2n})$ is 
     \[\displaystyle \begin{pmatrix}
-(n+1)\sqrt{2} & -(n-1)\sqrt{2} &  -(2n-1)\sqrt{2} \\
 n-1 & n-3 &  1\\
\end{pmatrix},\]
\noindent and the remaining eigenvalues are the roots of the following polynomial\\
$\Big[x - (2n-1)\sqrt{2}\Big] \Big[x - (n-1)(n+1)\sqrt{2}\Big] \Big[x - (n-1)(n-3)\sqrt{2}\Big]
 - 2(n-2)(5n^2 -6n + 2) \times$\\
\\
$\Big[x - (n-1)(n+1)\sqrt{2}\Big] 
 - 2n(5n^2 -2n +2) \Big[x - (n-1)(n-3)\sqrt{2}\Big]$.

\vspace{0.2cm}
\item[(iii)] If $n$ is even and $\frac{n}{2}$ is even, then the Sombor spectrum of conjugacy super commuting graph $\Delta^c(D_{2n})$ is
\[\displaystyle \begin{pmatrix}
-(\frac{n}{2}+1)\sqrt{2} & -(n-1)\sqrt{2} & -(2n-1)\sqrt{2} \\
 n-2 & n-3 &1 \\
\end{pmatrix},\]
\noindent and remaining eigenvalues are the roots of the following polynomial\\
$\Big[x - (2n-1)\sqrt{2}\Big] \Big[x - (n-1)(n-3)\sqrt{2}\Big] \Big[x - (\frac{n}{2}-1)(\frac{n}{2}+1)\sqrt{2}\Big]^2
 - 2(n-2)(5n^2 -6n +2) \times$\\
 \\
 $\Big[x - (\frac{n}{2}-1)(\frac{n}{2}+1)\sqrt{2}\Big]^2 
 - 2n \Big[(2n-1)^2 +(\frac{n}{2}+1)^2\Big] \Big[x - (n-3)(n-1)\sqrt{2}\Big]\Big[x - (\frac{n}{2}-1)(\frac{n}{2}+1)\sqrt{2}\Big]$. 
\end{itemize}
\end{corollary}

    


\section{Super enhanced power graph}

Aalipour et al. \cite{aalipour2017enhanced} proposed the idea of the enhanced power graph of a group to see how close the power graph is to the commuting graph. The enhanced power graph $\mathcal{P}_E(G)$ of a group $G$ is a simple graph whose vertex set is the whole group $G$ and two distinct vertices $x,y$ are adjacent if $x,y \in \left<z\right>$ for some $z \in G$. Further, the enhanced power graphs have been studied by various researchers. Bera and Bhuniya \cite{bera2018enhanced} characterized the abelian groups and the non-abelian $p$-groups, where $p$ is a prime, having dominatable enhanced power graphs. Dalal \emph{et al.} \cite{dalal2021enhanced} investigated the graph-theoretic properties of enhanced power graphs over the dicyclic group and the group $V_{8n}$. Additionally, Parveen \emph{et al.} \cite{parveen2024enhanced} explore the Laplacian spectrum of enhanced power graph of certain non-abelian groups. In this section, we discuss the Sombor spectrum of the super enhanced power graphs $\mathcal{P}_E(G), \mathcal{P}_E^o(G), \mathcal{P}_E^c(G)$, for $G= D_{2n}, Q_{4n}$ and $SD_{8n}$ into three subsections.


\subsection{Enhanced Power Graph}

In this subsection, we discuss the Sombor spectrum of the enhanced power graph $\mathcal{P}_E(G)$, for $G= D_{2n}, Q_{4n}$ and $SD_{8n}$. Parveen \emph{et al.} explored the structure of enhanced power graph $\mathcal{P}_E(D_{2n})$ and $\mathcal{P}_E(Q_{4n})$ (see \cite{parveen2024enhanced}) and their representation are  given below
\[
\mathcal{P}_E(D_{2n})=
    K_{1,2}[K_1,K_{n-1},\underbrace{K_1, \ldots, K_1}_{n-{\rm times}}],
\]

and 

\[
\mathcal{P}_E(Q_{4n}) = K_{1,n+1}[K_2, K_{2n-2}, \underbrace{K_2, \ldots, K_2}_{n-{\rm times}}].
\]

By Corollary \ref{Generalize join}, the proof of the following Corollaries are straightforward.

\begin{corollary} \label{Enhanced-D_2n}
Let $\mathcal{P}_{E}(D_{2n})$ be the enhanced power graph of the dihedral group. Then the  Sombor spectrum is 
    \[\displaystyle \begin{pmatrix}
-(n-1)\sqrt{2} & 0 \\
 n-2 & n-1 \\
\end{pmatrix},
\]
and the remaining three eigenvalues are the roots of the following  polynomial 

$ x^2\Big[x-(n-1)(n-2)\sqrt{2}\Big] 
 - (n-1)(5n^2 -6n +2)x 
 - 2n(2n^2-2n + 1)\Big[x-(n-1)(n-2)\sqrt{2}\Big].$
\end{corollary}



\begin{corollary} \label{Enhanced-Q_4n}
Let $\mathcal{P}_{E}(Q_{4n})$ be the enhanced power graph of the generalized quaternion group. Then the Sombor spectrum is

    \[\displaystyle \begin{pmatrix}
-(4n-1)\sqrt{2} & -(2n-1)\sqrt{2} & -3\sqrt{2} \\
 1 & 2n-3 & n  \\
\end{pmatrix},
\]
\noindent and the remaining eigenvalues are the roots of the following polynomial\\ 
$\Big[x-(4n-1)\sqrt{2}\Big] \Big[ x - (2n-1)(2n-3)\sqrt{2} \Big] \Big[x -3\sqrt{2}\Big]^n 
 - 8(n-1) (10n^2 -6n +1) \Big[x -3\sqrt{2}\Big]^n
- 8n(8n^2 - 4n +5) \times$ \\
\\
$\Big[x - (2n-1)(2n-3)\sqrt{2}\Big] \Big[x -3\sqrt{2}\Big]^{n-1}$.
\end{corollary} 

The following theorem will be useful in the sequel. 

\begin{lemma}[{\cite[Lemma 2.3]{parveen2024enhanced}}]\label{EPG-SD8n-NBD}
In $\mathcal{P}_E(SD_{8n})$, we have
\begin{enumerate}
\item[\rm (i)] {\rm N}$[e]= SD_{8n}$.
\item[\rm (ii)] {\rm N}$[a^{2n}]= \langle a \rangle \cup \{a^{2i + 1}b: 0 \leq i \leq 2n -1 \}$.
\item[\rm (iii)] {\rm N}$[a^{i}]= \langle a \rangle$, where $1 \leq i \leq 4n -1$ and $i \neq 2n$.
\item[\rm (iv)] {\rm N}$[a^{2i+ 1}b]= \langle a^{2i + 1}b \rangle = \{e, a^{2n},  a^{2i+ 1}b, a^{2n + 2i+ 1}b\}$, where $0 \leq i \leq 2n -1$.
\item[\rm (v)] {\rm N}$[a^{2i}b]= \{e, a^{2i}b\}$, where $1 \leq i \leq 2n$.
\end{enumerate}
\end{lemma}




\begin{theorem} \label{Enhanced-SD_8n}
Let $\mathcal{P}_{E}(SD_{8n})$ be the enhanced power graph of the semidihedral group. Then the Sombor spectrum is

\[\displaystyle \begin{pmatrix}
-(4n-1)\sqrt{2} & 0 & -3\sqrt{2} \\
 4n-3 & 2n-1 & n  \\
\end{pmatrix},
\]
and the remaining eigenvalues are the eigenvalues of the following equitable quotient matrix

\begin{equation*}
    \displaystyle
    \begin{bmatrix}
    0 & \sqrt{\alpha^2+\beta^2} & (4n-2)\sqrt{\alpha^2+\gamma^2} & 2\sqrt{\alpha^2+3^2} & \cdots &  2\sqrt{\alpha^2+3^2} &  2n\sqrt{\alpha^2+1}  \\
    \sqrt{\alpha^2+\beta^2} & 0 & (4n-2)\sqrt{\beta^2+\gamma^2} & 2\sqrt{\beta^2+3^2} & \cdots & 2\sqrt{\beta^2+3^2} & 0 \\
    \sqrt{\alpha^2+\gamma^2} & \sqrt{\beta^2+\gamma^2} & (4n-3)(4n-1)\sqrt{2} & 0 & \cdots &  0 & 0  \\
    \sqrt{\alpha^2+3^2} & \sqrt{\beta^2+3^2} & 0 & 3\sqrt{2} & \cdots & 0 & 0 \\
\vdots & \vdots & \vdots & \vdots & \ddots & \vdots &  \vdots \\
     \sqrt{\alpha^2+3^2} & \sqrt{\beta^2+3^2} & 0 & 0 & \cdots & 3\sqrt{2} &  0  \\
   \sqrt{\alpha^2+1} & 0 & 0 & 0 & \cdots & 0 & 0  
    \end{bmatrix}
    ,\end{equation*}
  where $\alpha=(8n-1), \; \beta=(6n-1)$ and $\gamma=(4n-1).$
\end{theorem}

\begin{proof}
First, we arrange the vertices of the enhanced power graph of the group $SD_{8n}$ in a sequence $e, a^{2n}, a, a^2, \ldots, a^{4n -1}$, $ab, a^{2n+1}b, a^3b,a^{2n+3}b, \ldots, a^{2i+1},a^{2n+ 2i+ 1}b,  \ldots, a^{2n-1}b, a^{4n-1}b,a^2b, a^4b, \ldots, a^{2i}b, \ldots, b$. By Lemma \ref{EPG-SD8n-NBD}, we have 
\[\mathcal{P}_{E}(SD_{8n})\cong K_1 \vee \Bigg( \Big\{K_1 \vee \Big[K_{4n-2} \cup \underbrace{K_{2} \cup K_2\cup\cdots \cup K_2}_{n-{\rm times}}\Big]\Big\} \cup \overline{K}_{2n}\Bigg).\]


Note that $d(e)=8n-1$, $d(a^{2n})=6n-1,\; d(a^i)=4n-1$, where $1 \leq i \leq 4n-1$ and $i \neq 2n$; $d(a^{2i + 1}b)=3$ for all $i$, where $0 \leq i \leq 2n-1$; and $d(a^{2i}b)=1$ for all $i$,  where $1 \leq i \leq 2n$ (cf. Lemma \ref{EPG-SD8n-NBD} ). Then the Sombor matrix of the enhanced power graph $\mathcal{P}_{E}(SD_{8n})$ is given below
\[
\displaystyle
\scriptsize
 \begin{bmatrix}
    0 & \sqrt{\alpha^2+\beta^2} & \sqrt{\alpha^2+\gamma^2} J_{1 \times 4n-2} & \sqrt{\alpha^2+3^2} J_{1 \times 2} &  \sqrt{\alpha^2+3^2} J_{1 \times 2} &\cdots&  
 \sqrt{\alpha^2+3^2} J_{1 \times 2} & \sqrt{\alpha^2+1} J_{1 \times 2n}  \\
    \sqrt{\alpha^2+\beta^2} & 0 & \sqrt{\beta^2+\gamma^2}J_{1 \times 4n-2} & \sqrt{\beta^2+3^2} J_{1 \times 2} & \sqrt{\beta^2+3^2} J_{1 \times 2} & \cdots & \sqrt{\beta^2+3^2} J_{1 \times 2} & \mathcal{O}_{1\times 2n} \\
    \sqrt{\alpha^2+\gamma^2}J_{4n-2 \times 1} & \sqrt{\beta^2+\gamma^2}J_{4n-2 \times 1} & \gamma\sqrt{2}(J_{4n-2}-I_{4n-2}) & \mathcal{O}_{4n-2\times 2} & \mathcal{O}_{4n-2\times 2} & \cdots & \mathcal{O}_{4n-2\times 2} & \mathcal{O}_{4n-2\times 2n} \\
    \sqrt{\alpha^2+3^2}J_{2 \times 1} & \sqrt{\beta^2+3^2}J_{2 \times 1} & \mathcal{O}_{2 \times 4n-2} & 3\sqrt{2}(J_{2}-I_{2}) & \mathcal{O}_{2 \times 2} & \cdots & \mathcal{O}_{2 \times 2} & \mathcal{O}_{2 \times 2n} \\
    \sqrt{\alpha^2+3^2}J_{2 \times 1} & \sqrt{\beta^2+3^2}J_{2 \times 1} & \mathcal{O}_{2 \times 4n-2} & \mathcal{O}_{2 \times 2} & 3\sqrt{2}(J_{2}-I_{2}) & \cdots & \mathcal{O}_{2 \times 2} & \mathcal{O}_{2 \times 2n} \\
    \vdots & \vdots & \vdots & \vdots & \vdots & \ddots & \vdots & \vdots \\
     \sqrt{\alpha^2+3^2}J_{2 \times 1} & \sqrt{\beta^2+3^2}J_{2 \times 1} & \mathcal{O}_{2 \times 4n-2} & \mathcal{O}_{2 \times 2} & \mathcal{O}_{2 \times 2} & \cdots & 3\sqrt{2}(J_{2}-I_{2}) & \mathcal{O}_{2 \times 2n} \\
 \sqrt{\alpha^2+1^2}J_{2n \times 1} & \mathcal{O}_{2n\times 1} & \mathcal{O}_{2n\times (4n-2)} & \mathcal{O}_{2n\times 2} & \mathcal{O}_{2n\times 2} & \cdots & \mathcal{O}_{2n\times 2} & \mathcal{O}_{2n\times 2n}  \\
\end{bmatrix}
,\]

where $J$ denotes the matrix in which each entry is one, $\mathcal{O}$ represent the zero matrix and $I$ is the identity matrix. Now, $K_{4n-2}, K_2, K_2,\ldots, K_2$  form cliques of order $4n-2$ and $2$, respectively and each vertex of these cliques share the same neighborhood. In view of Lemma \ref{multiplicity}, we have $-(4n-1)\sqrt{2}$ and $-3\sqrt{2}$ are the Sombor eigenvalues with multiplicities $4n-3$ and $n$, respectively. Also, $\overline{K}_{2n}$ forms an independent set and each vertex of this set shares the same neighborhood in the graph. Therefore, $0$ is an eigenvalue of $\mathcal{P}_E(SD_{8n})$ with multiplicity $2n-1$ (see  Lemma \ref{multiplicity}). Thus,  $7n-4$ eigenvalues are given below 
\[\displaystyle \begin{pmatrix}
-(4n-1)\sqrt{2} & 0 & -3\sqrt{2} \\
 4n-3 & 2n-1 & n  \\
\end{pmatrix},
\]
and the remaining $n + 4$ eigenvalues are the eigenvalues of the following equaitable quotient matrix
\[\begin{bmatrix}
    0 & \sqrt{\alpha^2+\beta^2} & (4n-2)\sqrt{\alpha^2+\gamma^2} & 2\sqrt{\alpha^2+3^2} & \cdots &  2\sqrt{\alpha^2+3^2} &  2n\sqrt{\alpha^2+1}  \\
    \sqrt{\alpha^2+\beta^2} & 0 & (4n-2)\sqrt{\beta^2+\gamma^2} & 2\sqrt{\beta^2+3^2} & \cdots & 2\sqrt{\beta^2+3^2} & 0 \\
    \sqrt{\alpha^2+\gamma^2} & \sqrt{\beta^2+\gamma^2} & (4n-3)(4n-1)\sqrt{2} & 0 & \cdots &  0 & 0  \\
    \sqrt{\alpha^2+3^2} & \sqrt{\beta^2+3^2} & 0 & 3\sqrt{2} & \cdots & 0 & 0 \\
\vdots & \vdots & \vdots & \vdots & \ddots & \vdots &  \vdots \\
     \sqrt{\alpha^2+3^2} & \sqrt{\beta^2+3^2} & 0 & 0 & \cdots & 3\sqrt{2} &  0  \\
   \sqrt{\alpha^2+1} & 0 & 0 & 0 & \cdots & 0 & 0  
    \end{bmatrix},\]
 where $\alpha=(8n-1), \; \beta=(6n-1)$ and $\gamma=(4n-1).$
\end{proof}

\begin{remark} For any finite group $G$, the order super enhanced power graph $\mathcal{P}_E^o(G)$ is equal to the order super commuting graph $\Delta^o(G)$ (see \cite{arunkumar2022super}). The Sombor spectrum of order super commuting graph $\Delta^o(G)$ has already been discussed in Subsection $4.2$. 
\end{remark}
 
\subsection{Conjugacy Super Enhanced Power Graph}
In this subsection, we discuss the Sombor spectrum of conjugacy super enhanced power graph of $D_{2n}, Q_{4n}, SD_{8n}.$

\begin{corollary} \label{conjugacySE-D_2n}
Let $\mathcal{P}_E^c(D_{2n})$ be the conjugacy super enhanced power graph of the dihedral group $D_{2n}$. 
\begin{itemize}
    \item [(i)] If $n$ is odd, then the Sombor spectrum of conjugacy super enhanced power graph $\mathcal{P}_E^c(D_{2n})$ is 
    \[\displaystyle \begin{pmatrix}
-(n-1)\sqrt{2} & -n\sqrt{2} \\
 n-2 & n-1 \\
\end{pmatrix},
\]
\noindent and the remaining eigenvalues are the roots of the following polynomial\\  
$x  \Big[x - (n-1)(n-2)\sqrt{2}\Big] \Big[x - n(n-1)\sqrt{2}\Big]
 - (n-1)(5n^2 -6n +2)  \Big[x - n(n-1)\sqrt{2}\Big]
 - n(5n^2 -4n + 1) \times$\\
 \\
 $\Big[x - (n-1)(n-2)\sqrt{2}\Big].$  
\vspace{0.2cm}
 \item[(ii)] If $n$ is even, then the Sombor spectrum of conjugacy super enhanced power graph $\mathcal{P}_E^c(D_{2n})$ is 
     \[\displaystyle \begin{pmatrix}
-\left(\frac{n}{2}\right)\sqrt{2} & -(n-1)\sqrt{2} & \frac{n(n-2)}{4}\\
 n-2 & n-2 & 1 \\
\end{pmatrix},\]
and the remaining three eigenvalues are the roots of the following polynomial\\
 $x \Big[x - (n-1)(n-2)\sqrt{2}\Big] \Big[\frac{4x - n(n-2)\sqrt{2}}{4}\Big]
 - (n-1)(5n^2 -6n +2) \Big[\frac{4x - n(n-2)\sqrt{2}}{4}\Big]
 - \frac{n}{4} (17n^2 -16n +4)\times$\\
 \\
 $\Big[x - (n-2)(n-1)\sqrt{2}\Big].$
\end{itemize}
\end{corollary}

\begin{proof} As we know the conjugacy classes of a dihedral group $D_{2n}$ are
\[\{e\}, \{a, a^{n-1}\}, \{a^2, a^{n-2}\}, \ldots, \{a^{\frac{n-1}{2}}, a^{\frac{n+1}{2}}\},\{b,ab, a^2b, \ldots,  a^{n-1}b\}~ \text{whenever}~ n~ \text{is~ odd}, \]
otherwise
\[\{e\}, \{ a^{\frac{n}{2}}\}, \{a, a^{n-1}\}, \{a^2, a^{n-2}\}, \ldots, \{a^{\frac{n-1}{2}}, a^{\frac{n+1}{2}}\},\{b,ab^3, a^5b, \ldots,  a^{n-1}b\}, \{a^2b, a^4b, \ldots,  a^{n-2}b\}. \]
For each conjugacy class, the subgraph induced by the vertices belonging to that conjugacy class is complete. In view of the structure of dihedral group $D_{2n}$ (see Equation \ref{eq(1)}), we get
\[
\mathcal{P}_E^c(D_{2n})= 
\begin{dcases}
K_{1,2}[K_1, K_{n-1},K_n],& \text{if $n$ is odd};\\
K_{1,3}[K_1, K_{n-1},K_{\frac{n}{2}},K_{\frac{n}{2}}],& \text{if $n$ is even}.\\
\end{dcases}
\]
By Corollary \ref{Generalize join}, we obtain the required result.
\end{proof}

\begin{corollary} \label{conjugacySE-Q_4n}
Let $\mathcal{P}_E^c(Q_{4n})$ be the conjugacy super enhanced power graph of the generalized quaternion group. 
\begin{itemize}
    \item [(i)] If $n$ is odd, then the Sombor spectrum of conjugacy super enhanced power graph $\mathcal{P}_E^c(Q_{4n})$ is  
    \[\displaystyle \begin{pmatrix}
-(4n-1)\sqrt{2} & -(2n-1)\sqrt{2} & -(2n+1)\sqrt{2} \\
 1 & 2n-3 & 2n-1  \\
\end{pmatrix},
\]
and the remaining eigenvalues are the roots of the following polynomial\\ 
$\Big[x-(4n-1)\sqrt{2}\Big] \Big[x - (2n-1)(2n-3)\sqrt{2}\Big] \Big[x - (2n-1)(2n+1)\sqrt{2}\Big]
 - 8(n-1) (10n^2 -6n + 1) \Big[x - (2n-1)(2n+1)\sqrt{2}\Big]$\\
\\
$-8n(10n^2 -2n + 1) \Big[x - (2n-1)(2n-3)\sqrt{2}\Big].$ 

\vspace{0.2cm}
\item[(ii)] If $n$ is even, then the Sombor spectrum of conjugacy super enhanced power graph $\mathcal{P}_E^c(Q_{4n})$ is
     \[\displaystyle \begin{pmatrix}
-(4n-1)\sqrt{2} & -(2n-1)\sqrt{2} & -(n+1)\sqrt{2} & (n^2-1)\sqrt{2} \\
 1 & 2n-3 & 2n-2 & 1  \\
\end{pmatrix},\]
and the remaining eigenvalues are roots of the following polynomial\\
$\Big[x - (4n-1)\sqrt{2}\Big] \Big[x - (n-1)(n+1)\sqrt{2}\Big] \Big[x - (2n-1)(2n-3)\sqrt{2}\Big]
 - 8(n-1) (10n^2 -6n + 1) \Big[x - (n-1)(n+1)\sqrt{2}\Big]$\\
 \\$- 4n(17n^2 -6n +2) \Big[x - (2n-3)(2n-1)\sqrt{2}\Big]$.
\end{itemize}
\end{corollary}

\begin{proof} As we know the conjugacy classes of a dihedral group $Q_{4n}$ are
\begin{itemize}
\item $[e]_c = \{e\}; [a^n]_c = \{a^n\}$;
\item for $0 < i < 2n$ and $i \neq n$, we have  $[a^i]_c = \{a^{i}, a^{-i}\}$;
\item $[a^2b]_c = \{a^{2i}b : 0 \leq i < n\}$;
\item $[ab]_c = \{a^{2i + 1}b : 0 \leq i < n\}$.
\end{itemize}
For each conjugacy class, the subgraph induced by the vertices belonging to that conjugacy class is complete. In view of the structure of dicyclic group $Q_{4n}$ (see Equation \ref{eq(2)}), we get
\[
\mathcal{P}_E^c(Q_{4n})= 
\begin{dcases}
K_{1,3}[K_2,K_{2n-2},K_{2n}],& \text{if $n$ is odd};\\
K_{1,2}[K_2,K_{2n-2},K_{n},K_{n} ] ,& \text{if $n$ is even}.
\end{dcases}
\]
By Corollary \ref{Generalize join}, we obtain the required result.
\end{proof}







\begin{theorem}\label{conjugacySE-SD_8n}
Let $\mathcal{P}_E^c(SD_{8n})$ be the conjugacy super enhanced power graph of the semidihedral group.  
\begin{itemize}
\item[(i)] If $n$ is odd, then the Sombor spectrum of conjugacy super enhanced power graph $\mathcal{P}_E^c(SD_{8n})$ is 
     \[\displaystyle \begin{pmatrix}
-(2n+1)\sqrt{2} & -(4n-1)\sqrt{2} & -n\sqrt{2}\\
 2n-1 & 4n-3  & 2n-2 \\
\end{pmatrix},\]
and the remaining eigenvalues are the eigenvalues of the following equitable  quotient matrix 
\[\begin{bmatrix}
    0 & \sqrt{{\alpha}^2+{\beta}^2} & (4n-2)\sqrt{{\alpha}^2+{\gamma}^2}  & (2n)\sqrt{{\alpha}^2+{\delta}^2}  & n\sqrt{{\alpha}^2+ {\psi}^2}  & n\sqrt{{\alpha}^2+{\psi}^2}  \\
    \sqrt{{\alpha}^2+{\beta}^2} & 0 & (4n-2)\sqrt{{\beta}^2+{\gamma}^2} & (2n)\sqrt{{\beta}^2+{\delta}^2}  & 0 & 0 \\
    \sqrt{{\alpha}^2+{\gamma}^2} & \sqrt{{\beta}^2+{\gamma}^2} & (4n-1)(4n-3)\sqrt{2} & 0 & 0 & 0 \\
    \sqrt{{\alpha}^2+{\delta}^2} & \sqrt{{\beta}^2+{\delta}^2} & 0 & (2n-1)(2n+1)\sqrt{2} & 0 & 0 \\
 \sqrt{{\alpha}^2+{\psi}^2} & 0 & 0 & 0 & n(n-1)\sqrt{2} & 0 \\
 \sqrt{{\alpha}^2+{\psi}^2} & 0 & 0 & 0 & 0 & n(n-1)\sqrt{2} 
\end{bmatrix}
,\]
where $\alpha = 8n-1,~ \beta=6n-1,~ \gamma=4n-1,~ \delta=2n+1$ and $\psi=n.$
\item[(ii)] If $n$ is even, then the Sombor spectrum of conjugacy super enhanced power graph $\mathcal{P}_E^c(SD_{8n})$ is 
    \[\displaystyle \begin{pmatrix}
-(4n-1)\sqrt{2} & -(2n+1)\sqrt{2} & -2n\sqrt{2} \\
 4n-3 & 2n-1 & 2n-1\\
\end{pmatrix},
\]
and the remaining eigenvalues are the eigenvalues of the following equitable quotient matrix
\[\begin{bmatrix}
    0 & \sqrt{\alpha^2+\beta^2} & (4n-2)\sqrt{\alpha^2+\gamma^2}  & 2n\sqrt{\alpha^2+\delta^2}  & 2n\sqrt{\alpha^2+\psi^2} \\
    \sqrt{\alpha^2+\beta^2} & 0 & (4n-2)\sqrt{\beta^2+\gamma^2} & 2n\sqrt{\beta^2+\delta^2}  & 0 \\
    \sqrt{\alpha^2+\gamma^2} & \sqrt{\beta^2+\gamma^2} & (4n-1)(4n-3)\sqrt{2} & 0 & 0 \\
    \sqrt{\alpha^2+\delta^2} & \sqrt{\beta^2+\delta^2} & 0 & (2n-1)(2n+1)\sqrt{2} & 0  \\
 \sqrt{\alpha^2+\psi^2} & 0 & 0 & 0 & 2n(2n-1)\sqrt{2} 
\end{bmatrix},\]
where $\alpha = 8n-1,~ \beta=6n-1,~ \gamma=4n-1,~ \delta=2n+1$ and $\psi= 2n.$
\end{itemize}
\end{theorem} 

\begin{proof}(i) In order to obtain the structure of conjugacy super enhanced power graph $P_E^c(SD_{8n})$, it is known that the conjugacy classes of semidihedral group $SD_{8n}$ are
\begin{itemize}
\item $[e]_c = \{e\}$;
\item $[a^{n}]_c = \{a^{n}\}$;
\item $[a^{2n}]_c = \{a^{2n}\}$;
\item $[a^{3n}]_c = \{a^{3n}\}$;
\item for $1 \leq i \leq 2n + 1$ and $i \neq \{n, 2n-1, 2n\}$, we have $[a^i]_c = \begin{dcases}
\{a^i, a^{4n -i}\},& \text{ if $i$ is even};\\
\{a^i, a^{2n -i}\},& \text{if $i$ is odd};
\end{dcases}$

\item $[b]_c = \{a^{4k}b:~ 0 \leq k \leq n-1 \}$;

\item $[ab]_c = \{a^{4k + 1}b:~ 0 \leq k \leq n-1 \}$;

\item $[a^2b]_c = \{a^{4k + 2}b:~ 0 \leq k \leq n-1 \}$;

\item $[a^3b]_c = \{a^{4k + 3}b:~ 0 \leq k \leq n-1 \}$.
\end{itemize}
For each conjugacy class,  the subgraph induced by the vertices belonging to that conjugacy class is complete. In view of the structure of the enhanced power graph $P_E(SD_{8n})$ given in the proof of Theorem \ref{Enhanced-SD_8n}, we have
\[\mathcal{P}_E^c(SD_{8n})= K_1\lor \Big(\Big[K_1\lor(K_{4n-2}\cup K_{2n})\Big]\cup K_n\cup K_n \Big).\]
\noindent Now, we arrange the vertices of the conjugacy super enhanced power graph $P_E^c(SD_{8n})$ in a sequence $e, a^{2n}, a, a^2$\\
$,\ldots, a^{2n-1}, a^{2n + 1}, \ldots, a^{4n -1}, ab, a^{2n +1}b, a^3b, a^{2n +3}b, \ldots, a^{2n -1}b, a^{4n -1}b, a^2b, a^6b, a^{10}b, \ldots, a^{2n}b, a^4b, a^8b, \ldots,b.$ Further, note that $d(e) = 8n-1;~ d(a^{2n}) = 6n - 1;~ d(a^i) = 4n - 1$ for all $1 \leq i \leq 4n -1$ and $i \neq 2n;~ d(a^{2i + 1}b) = 2n + 1$ for all $0 \leq i \leq 2n -1$ and $d(a^{2i}b) = n$ for all $0 \leq i \leq 2n -2$. Then the Sombor matrix of $\mathcal{P}_E^c(SD_{8n})$ is
\[
 \displaystyle
\begin{bmatrix}
    0 & \sqrt{\alpha^2+\beta^2} & \sqrt{{\alpha}^2+{\gamma}^2} J_{1 \times 4n-2} & \sqrt{\alpha^2+\delta^2} J_{1 \times 2n} & \sqrt{\alpha^2+\psi^2} J_{1 \times 2n} \\
    \sqrt{\alpha^2+\beta^2} & 0 & \sqrt{\beta^2+\gamma^2}J_{1 \times 4n-2} & \sqrt{\beta^2+\delta^2} J_{1 \times 2n} & \mathcal{O}_{1\times 2n} \\
    
    \sqrt{{\alpha}^2+{\gamma}^2}J_{4n-2 \times 1} & \sqrt{\beta^2+\gamma^2}J_{4n-2 \times 1} & \gamma \sqrt{2}(J_{4n-2}-I_{4n-2}) & \mathcal{O}_{4n-2\times 2n} & \mathcal{O}_{4n-2\times 2n}\\
    \sqrt{\alpha^2+\delta^2}J_{2n \times 1} & \sqrt{\beta^2+\delta^2}J_{2n \times 1} & \mathcal{O}_{2n\times 4n-2} & \delta\sqrt{2}(J_{2n}-I_{2n}) & \mathcal{O}_{2n \times 2n}\\
 \sqrt{\alpha^2+\psi^2}J_{2n \times 1} & \mathcal{O}_{2n \times 1} & \mathcal{O}_{2n \times 4n-2} & \mathcal{O}_{2n\times 2n} & \psi \sqrt{2}(J_{2n}-I_{2n})\\
  \end{bmatrix}
,\]

where $J$ denotes the matrix in which each entry is one, $\mathcal{O}$ represents the zero matrix, and $I$ is the identity matrix with
$\alpha = 8n-1, \beta = 6n-1, \gamma = 4n-1, \delta = 2n+1$ and $\psi = n.$
Now, $K_{4n-2}, K_{2n}$ and $K_{2n}$ form cliques of orders $4n-2$ and $2n$(2 times), respectively, and each vertex of these cliques share the same neighborhood. By Lemma $2.4$, we have $-(4n-1)\sqrt{2}, -(2n+1)\sqrt{2}$ and $-2n\sqrt{2}$ are the Sombor eigenvalues with multiplicities $4n-3, 2n-1$ and $2n-1$, respectively. Therefore, $8n-5$ eigenvalues are known. 

\[\displaystyle \begin{pmatrix}
-(4n-1)\sqrt{2} & -(2n+1)\sqrt{2} & -n\sqrt{2} \\
 4n-3 & 2n-1 & 2n-1  \\
\end{pmatrix},
\]
and the remaining eigenvalues are the eigenvalues of the following equitable quotient matrix.

\[\begin{bmatrix}
    0 & \sqrt{{\alpha}^2+{\beta}^2} & (4n-2)\sqrt{{\alpha}^2+{\gamma}^2}  & (2n)\sqrt{{\alpha}^2+{\delta}^2}  & n\sqrt{{\alpha}^2+ {\psi}^2}  & n\sqrt{{\alpha}^2+{\psi}^2}  \\
    \sqrt{{\alpha}^2+{\beta}^2} & 0 & (4n-2)\sqrt{{\beta}^2+{\gamma}^2} & (2n)\sqrt{{\beta}^2+{\delta}^2}  & 0 & 0 \\
    \sqrt{{\alpha}^2+{\gamma}^2} & \sqrt{{\beta}^2+{\gamma}^2} & (4n-1)(4n-3)\sqrt{2} & 0 & 0 & 0 \\
    \sqrt{{\alpha}^2+{\delta}^2} & \sqrt{{\beta}^2+{\delta}^2} & 0 & (2n-1)(2n+1)\sqrt{2} & 0 & 0 \\
 \sqrt{{\alpha}^2+{\psi}^2} & 0 & 0 & 0 & n(n-1)\sqrt{2} & 0 \\
 \sqrt{{\alpha}^2+{\psi}^2} & 0 & 0 & 0 & 0 & n(n-1)\sqrt{2} 
\end{bmatrix}
.\]

(ii) To find the representation of conjugacy super enhanced power graph $P_E^c(SD_{8n})$, we examine the conjugacy classes of semidihedral group $SD_{8n}$. These conjugacy classes are given as follows:
\begin{itemize}
\item $[e]_c = \{e\}$;
\item $[a^{2n}]_c = \{a^{2n}\}$;
\item for $1 \leq i \leq 2n + 1$ and $i \neq \{2n-1, 2n\}$, we have $[a^i]_c = \begin{dcases}
\{a^i, a^{4n -i}\},& \text{ if $i$ is even};\\
\{a^i, a^{2n -i}\},& \text{if $i$ is odd};
\end{dcases}$
\item $[ab]_c = \{a^{2k + 1}b:~ 0 \leq k \leq 2n-1 \}$;
\item $[b]_c = \{a^{2k}b:~ 0 \leq k \leq 2n-1 \}$.
\end{itemize}
For each conjugacy class,  the subgraph induced by the vertices belonging to that conjugacy class is complete. In view of the structure of enhanced power graph $P_E(SD_{8n})$ given in the proof of Theorem \ref{Enhanced-SD_8n}, we have
\[\mathcal{P}_E^c(SD_{8n})= K_1\lor \Big(\Big[K_1\lor(K_{4n-2}\cup K_{2n})\Big]\cup K_{2n}\Big).\]
\noindent Now, we arrange the vertices of the conjugacy super enhanced power graph $P_E^c(SD_{8n})$ in a sequence $e, a^{2n}, a, a^2$\\
$,\ldots, a^{2n-1}, a^{2n + 1}, \ldots, a^{4n -1}, ab, a^{2n +1}b, a^3b, a^{2n +3}b, \ldots, a^{2n -1}b, a^{4n -1}b, a^2b, a^6b, a^{10}b, \ldots, a^{2n}b,$ 
 and then $b, a^4b, a^8b$ $,\\ \ldots, a^{4n -4}b$. Further, note that $d(e) = 8n-1;~ d(a^{2n}) = 6n - 1;~ d(a^i) = 4n - 1$ for all $1 \leq i \leq 4n -1$ and $i \neq 2n;~ d(a^{2i + 1}b) = 2n + 1$ for all $0 \leq i \leq 2n -1$ and $d(a^{2i}b) = n$ for all $0 \leq i \leq 2n -2$. Then the Sombor matrix of $\mathcal{P}_E^c(SD_{8n})$ is
\[
 \displaystyle
\begin{bmatrix}
    0 & \sqrt{\alpha^2+\beta^2} & \sqrt{{\alpha}^2+{\gamma}^2} J_{1 \times 4n-2} & \sqrt{\alpha^2+\delta^2} J_{1 \times 2n} & \sqrt{\alpha^2+\psi^2} J_{1 \times 2n} \\
    \sqrt{\alpha^2+\beta^2} & 0 & \sqrt{\beta^2+\gamma^2}J_{1 \times 4n-2} & \sqrt{\beta^2+\delta^2} J_{1 \times 2n} & \mathcal{O}_{1\times 2n} \\
    
    \sqrt{{\alpha}^2+{\gamma}^2}J_{4n-2 \times 1} & \sqrt{\beta^2+\gamma^2}J_{4n-2 \times 1} & \gamma \sqrt{2}(J_{4n-2}-I_{4n-2}) & \mathcal{O}_{4n-2\times 2n} & \mathcal{O}_{4n-2\times 2n}\\
    \sqrt{\alpha^2+\delta^2}J_{2n \times 1} & \sqrt{\beta^2+\delta^2}J_{2n \times 1} & \mathcal{O}_{2n\times 4n-2} & \delta\sqrt{2}(J_{2n}-I_{2n}) & \mathcal{O}_{2n \times 2n}\\
 \sqrt{\alpha^2+\psi^2}J_{2n \times 1} & \mathcal{O}_{2n \times 1} & \mathcal{O}_{2n \times 4n-2} & \mathcal{O}_{2n\times 2n} & \psi \sqrt{2}(J_{2n}-I_{2n})\\
  \end{bmatrix}
,\]

where $J$ denotes the matrix in which each entry is one, $\mathcal{O}$ represents the zero matrix, and $I$ is the identity matrix with
$\alpha = 8n-1, \beta = 6n-1, \gamma = 4n-1, \delta = 2n+1$ and $\psi = 2n.$
Now, $K_{4n-2}, K_{2n}$ and $K_{2n}$ form cliques of orders $4n-2$ and $2n$(2 times), respectively, and each vertex of these cliques share the same neighborhood. By Lemma \ref{multiplicity}, we have $-(4n-1)\sqrt{2}, -(2n+1)\sqrt{2}$ and $-2n\sqrt{2}$ are the Sombor eigenvalues with multiplicities $4n-3, 2n-1$ and $2n-1$, respectively. Therefore, $8n-5$ eigenvalues are known. 

\[\displaystyle \begin{pmatrix}
-(4n-1)\sqrt{2} & -(2n+1)\sqrt{2} & -n\sqrt{2} \\
 4n-3 & 2n-1 & 2n-1  \\
\end{pmatrix},
\]
and the remaining eigenvalues are the eigenvalues of the following equitable quotient matrix

\[\begin{bmatrix}
    0 & \sqrt{\alpha^2+\beta^2} & (4n-2)\sqrt{\alpha^2+\gamma^2}  & 2n\sqrt{\alpha^2+\delta^2}  & 2n\sqrt{\alpha^2+\psi^2} \\
    \sqrt{\alpha^2+\beta^2} & 0 & (4n-2)\sqrt{\beta^2+\gamma^2} & 2n\sqrt{\beta^2+\delta^2}  & 0 \\
    \sqrt{\alpha^2+\gamma^2} & \sqrt{\beta^2+\gamma^2} & (4n-1)(4n-3)\sqrt{2} & 0 & 0 \\
    \sqrt{\alpha^2+\delta^2} & \sqrt{\beta^2+\delta^2} & 0 & (2n-1)(2n+1)\sqrt{2} & 0  \\
 \sqrt{\alpha^2+\psi^2} & 0 & 0 & 0 & 2n(2n-1)\sqrt{2} 
\end{bmatrix}.\]

\end{proof}

\section{Super power graph}
In $2002$, Kelarev and Quinn \cite{a.kelarev2002undirected} introduced the notion of a directed power graph $\vec{\mathcal{P}}(G)$ of a group $G$. Inspired by this, Chakrabarty \emph{et al.} \cite{chakrabarty2009undirected} introduced the concept of an undirected power graph of a semigroup $G$, which was defined as follows: Given a group $G$, the power graph of $G$ is the simple undirected graph with a set of vertex $G$ and two vertices $x$, $y$ in $G$ are adjacent in the power graph $\mathcal{P}(G)$ if and only if one of them is a power of the other. Various researchers have contributed to the power graph $\mathcal{P}(G)$ (see \cite{cameron2011power, cameron2020connectivity, dalal2023connectivity,a.2017Hamzaeh_automorphism,shitov2017coloring}).  For more detail on the power graph, we refer the reader to \cite{kumar2021recent}. In Figures $1-9$, a subset $T$ of the vertex set $G$ is shown in a red circle which means that the subgraph induced by the vertices belongs to the set $T$ and subset $T'$ of the vertex set $G$ is shown in a black circle which means that it is a complete subgraph whose vertices belongs to $T'$. The  following theorem characterizes the group $G$ such that the power graph $\mathcal{P}(G)$ is equal to the order super power graph $\mathcal{P}^o(G)$.

\begin{theorem}{\cite[Theorem 3.1]{dalal2024reducedgraph}}\label{P(G)=P^o(G)}
Let $G$ be a finite group. Then the following are equivalent:
\begin{enumerate}
\item[\rm (i)] $\mathcal{P}^o(G) =\mathcal{P}(G)$;
\item[\rm (ii)] $\mathcal{P}_E^o(G) =\mathcal{P}_E(G)$;
\item[\rm (iii)] $G$ is cyclic.
\end{enumerate}
\end{theorem}

In this section, we discuss the Sombor spectrum of the super power graphs $\mathcal{P}(G), \mathcal{P}^o(G), \mathcal{P}^c(G)$, for $G= D_{2n}, Q_{4n}$ and $SD_{8n}$ into three subsections.

\subsection{Power graph}
This subsection is devoted to the Sombor spectrum of the power graphs $\mathcal{P}(G)$, for $G= D_{2n}, Q_{4n}$ and $SD_{8n}$.
The following lemma will be useful to obtain the Sombor spectrum of the power graph of the dihedral group $D_{2n}$.

\begin{lemma}\label{P(D2n)-structure}
Let $G$ be a dihedral group $D_{2n}$. Then 
\[
\mathcal{P}(D_{2n}) = K_1 \vee \left[ K_{\phi(n)} \vee \Gamma_{n}[K_{\phi(d_1)},K_{\phi(d_2)},\ldots,K_{\phi(d_t)}] \cup \overline{K}_n \right],
\]
where $\Gamma_n$ is a graph with vertex $V(\Gamma_n) = \{ d_i : 1, n \neq d_i | n,~ 1 \leq i \leq t \}$ and two distinct  vertices $d_i$ and $d_j$ are adjacent in $\Gamma_n$ if one of them divides other.
\end{lemma}
\begin{proof} In view of the representation of the dihedral group $D_{2n}$, we observe that $a^ib \nsim a^j$ for all $1 \leq i, j \leq n$ and $j \neq n$. Also, the subgraph induced by the vertices belongs to the cyclic subgroup generated by the element $a$, $\mathcal{P}(\mathbb Z_n) =  K_{\phi(n) + 1} \vee \Gamma_{n}[K_{\phi(d_1)}, K_{\phi(d_2)},\ldots,K_{\phi(d_t)}]$ (see \cite{mehranian2016note}). Thus, we get the required representation of $\mathcal{P}(D_{2n})$.
\end{proof}

\begin{theorem}\label{power D_2n}
 Let $G$ be a dihedral group $D_{2n}$. Then  the Sombor spectrum of power graph $\mathcal{P}(D_{2n})$ is
 \[\displaystyle \begin{pmatrix}
-(n-1)\sqrt{2} & -\bar{d}_1\sqrt{2}  & -\bar{d}_2\sqrt{2} & \cdots & -\bar{d}_t\sqrt{2} & 0 \\
 \phi(n)-1 & \phi(d_1)-1 & \phi(d_2)-1 & & \phi(d_t)-1 & n-1  \\
\end{pmatrix},
\]
and the remaining eigenvalues are the eigenvalues of the following equaitable quotient matrix 

\[
 \displaystyle
 \begin{bmatrix}
0 & \phi(n)\sqrt{\alpha^2+\beta^2} & \phi(d_1)\sqrt{\alpha^2+\bar{d_1}^2}  & \phi(d_2)\sqrt{\alpha^2+\bar{d_2}^2} & \cdots & \phi(d_1)\sqrt{\alpha^2+\bar{d_t}^2}  & n\sqrt{\alpha^2+1^2}\\
\sqrt{\alpha^2+\beta^2} & \beta(\phi(n)-1)\sqrt{2}& \phi(d_1)\sqrt{\beta^2+\bar{d_1}^2} & \phi(d_2)\sqrt{\beta^2+\bar{d_2}^2} & \cdots & \phi(d_t) \sqrt{\beta^2+ \bar{d_t}^2} & 0\\
\sqrt{\alpha^2+ \bar{d_1}^2} & \phi(n)\sqrt{\beta^2+ \bar{d_1}^2} & (\phi(d_1)-1)\bar{d_1}\sqrt{2} & \phi(d_2)a_{12} & \cdots & \phi(d_t)a_{1t} & 0 \\
\sqrt{\alpha^2+ \bar{d_2}^2} & \phi(n)\sqrt{\beta^2+ \bar{d_2}^2} & \phi(d_1)a_{21} & (\phi(d_2)-1)\bar{d_2}\sqrt{2} & \cdots & \phi(d_t)a_{2t} & 0 \\
\vdots & \vdots & \vdots & \vdots & \ddots & \vdots & \vdots\\
\sqrt{\alpha^2+ \bar{d_t}^2} & \phi(n)\sqrt{\beta^2+\overline{d_t}^2} & \phi(d_1)a_{t1} & \phi(d_2)a_{t2} & \cdots & (\phi(d_t)-1)\bar{d_t}\sqrt{2} & 0 \\
\sqrt{\alpha^2+1^2} & 0 & 0 & 0 & \cdots & 0 & 0 \\
\end{bmatrix}
,\]
  where $\alpha=(2n-1), \; \beta=(n-1)$, $\bar{d_i}= d(a^j)$ for some $a^j$ with $o(a^j)= d_i$, and $a_{ij}=
\begin{dcases}
 \sqrt{\bar{d_i}^2+\bar{d_j}^2}, & d_i \mid d_j; \\
 0, & otherwise.
\end{dcases}$
\end{theorem}

\begin{proof} By Lemma \ref{P(D2n)-structure}, we have 
\[
\mathcal{P}(D_{2n}) = K_1 \vee \left[ K_{\phi(n)} \vee \Gamma_{n}[K_{\phi(d_1)},K_{\phi(d_2)},\ldots,K_{\phi(d_t)}] \cup \overline{K}_n \right],
\]
where $\Gamma_n$ is a graph with vertex $V(\Gamma_n) = \{ d_i : 1, n \neq d_i | n,~ 1 \leq i \leq t \}$ and two distinct  vertices $d_i$ and $d_j$ are adjacent in $\Gamma_n$ if one of them divides other. Clearly, $d(e) = 2n-1;~ d(a^ib) = 1$ for all $1 \leq i \leq n$; and for each $1 \leq i \leq n - 1$, we have $d(a^i) = \phi(d_j)$  for some $1 \leq j \leq t$ (see \cite{mehranian2016note}). Therefore, we consider $e \in V(K_1), a^ib \in V(\overline{K}_n)$ for all $1 \leq i \leq n$, and $a^i \in \Gamma_{n}[K_{\phi(d_1)},K_{\phi(d_2)},\ldots,K_{\phi(d_t)}]$ for all $1 \leq i \leq n - 1$. Moreover, for $1 \leq i_r \leq n$ and $(i_r, n) = n$, let $a^{^{i_r}} \in V(K_{\phi(n)})$. Also, for each $1 \leq i \leq n-1$ with $o(a^i) = d_i$, we assume $a^i \in V(K_{d_i})$. Note that the degree of all the vertices belonging to $K_{d_i}$ are equal and we assume that it is $\bar{d_i}$ for some positive integer $\bar{d_i}$, where $1 \leq i \leq t$. The Sombor matrix of $\mathcal{P}(D_{2n})$ is the $2n \times 2n$ matrix given below, where the rows and columns are indexed in order by the vertices $e, a^{i_2}, a^{i_3}, \ldots, a^{i_{\phi(n)}}, a^{^{d^1_1}}, a^{^{d^2_1}}, \ldots, a^{^{d^{\phi(d_1)}_1}}, \ldots\ldots, a^{^{d^1_t}}, a^{^{d^2_t}}, \ldots, a^{^{d^{\phi(d_t)}_t}},$ and then $b, ab, a^2b, \ldots, a^{n-1}b$.

\[
 \displaystyle
 \scriptsize
    \begin{bmatrix}
0 & \sqrt{\alpha^2+\beta^2}J_{1 \times \phi(n)} & \sqrt{\alpha^2+\bar{d_1}^2} J_{1 \times \phi(d_1)} & \sqrt{\alpha^2+\bar{d_2}^2} J_{1 \times \phi(d_2)} & \cdots & \sqrt{\alpha^2+\bar{d_t}^2} J_{1 \times \phi(t)} & \sqrt{\alpha^2+1^2} J_{1 \times n}\\
\sqrt{\alpha^2+\beta^2} & \beta\sqrt{2}(J_{\phi(n)}-I_{\phi(n)}) & \sqrt{\beta^2+\bar{d_1}^2}J_{\phi(n) \times \phi(d_1)} & \sqrt{\beta^2+\bar{d_2}^2}J_{\phi(n) \times \phi(d_2)} & \cdots & \sqrt{\beta^2+\bar{d_t}^2}J_{\phi(n) \times \phi(d_t)} &\mathcal{O}_{\phi(n) \times n} \\
\sqrt{\alpha^2+\bar{d_1}^2}J_{\phi(d_1) \times 1} & \sqrt{\beta^2+\bar{d_1}^2}J_{\phi(d_1) \times \phi(n)} & \bar{d_1}\sqrt{2}(J_{\phi(d_1)}-I_{\phi(d_1)}) & \mathcal{A}_{12} & \cdots & \mathcal{A}_{1t} & \mathcal{O}_{\phi(d_1) \times n} \\
\sqrt{\alpha^2+\bar{d_2}^2}J_{\phi(d_2) \times 1} & \sqrt{\beta^2+\bar{d_2}^2}J_{\phi(d_2) \times \phi(n)} & \mathcal{A}_{21} &  \bar{d_2}\sqrt{2}(J_{\phi(d_2)}-I_{\phi(d_2)}) & \cdots & \mathcal{A}_{2t} & \mathcal{O}_{\phi(d_2) \times n} \\
\vdots & \vdots & \vdots & \vdots & \ddots & \vdots & \vdots\\
\sqrt{\alpha^2+\bar{d_t}^2}J_{\phi(d_t) \times 1} & \sqrt{\beta^2+\bar{d_t}^2}J_{\phi(d_t) \times \phi(n)} & \mathcal{A}_{t1} & \mathcal{A}_{t2} & \cdots & \bar{d_t}\sqrt{2}(J_{\phi(d_t)}-I_{\phi(d_2)})& \mathcal{O}_{\phi(d_t) \times n} \\
\sqrt{\alpha^2+1^2}J_{n \times 1} & \mathcal{O}_{n \times \phi(n)} & \mathcal{O}_{n \times\phi(d_1)} & \mathcal{O}_{n \times\phi(d_2)} & \cdots & \mathcal{O}_{n \times\phi(d_t)}& \mathcal{O}_{n \times n} \\
\end{bmatrix}
,\]
 with
$\alpha= 2n-1, \beta=n-1$ and
$\mathcal{A}_{ij}=\mathcal{A}_{ji}= [a_{ij}]_{\phi(d_i) \times\phi(d_j)}$, 
where
\[ a_{ij}=
\begin{dcases}
 \sqrt{\bar{d_i}^2+\bar{d_j}^2}, & d_i \mid d_j; \\
 0, & otherwise.
\end{dcases}
\]

Now, $K_{\phi(n)}, K_{\phi(d_1)}, K_{\phi(d_2)},\ldots, K_{\phi(d_t)}$  form cliques of order $\phi(n), \phi(d_1), \phi(d_2),\ldots,\phi(d_t)$, respectively and each vertex of these cliques share the same neighborhood. So by Lemma $2.4$, we have $-(n-1)\sqrt{2}, -\overline{d}_1\sqrt{2},-\overline{d}_2\sqrt{2},\ldots,  -\overline{d}_t\sqrt{2}$ are the Sombor eigenvalues with multiplicities $\phi(n), \phi(d_1), \phi(d_2), \ldots, \phi(d_t)$, respectively. Also, $\overline{K}_{n}$ forms an independent set and share the same neighborhood in the graph. By Lemma \ref{multiplicity}, we conclude that $0$ is the Sombor eigenvalue with multiplicity $n-1.$ It implies that, from the total $2n$ eigenvalues $2n-(t+3)$ are known 

 \[\displaystyle \begin{pmatrix}
-(n-1)\sqrt{2} & -\bar{d}_1\sqrt{2}  & -\bar{d}_2\sqrt{2} & \cdots & -\bar{d}_t\sqrt{2} & 0 \\
 \phi(n)-1 & \phi(d_1)-1 & \phi(d_2)-1 & & \phi(d_t)-1 & n-1  \\
\end{pmatrix},
\]
and the remaining eigenvalues are the eigenvalues of the equaitable quotient matrix which is given below 
\[
 \displaystyle
 \begin{bmatrix}
0 & \phi(n)\sqrt{\alpha^2+\beta^2} & \phi(d_1)\sqrt{\alpha^2+\bar{d_1}^2}  & \phi(d_2)\sqrt{\alpha^2+\bar{d_2}^2} & \cdots & \phi(d_1)\sqrt{\alpha^2+\bar{d_t}^2}  & n\sqrt{\alpha^2+1^2}\\
\sqrt{\alpha^2+\beta^2} & \beta(\phi(n)-1)\sqrt{2}& \phi(d_1)\sqrt{\beta^2+\bar{d_1}^2} & \phi(d_2)\sqrt{\beta^2+\bar{d_2}^2} & \cdots & \phi(d_t) \sqrt{\beta^2+ \bar{d_t}^2} & 0\\
\sqrt{\alpha^2+ \bar{d_1}^2} & \phi(n)\sqrt{\beta^2+ \bar{d_1}^2} & (\phi(d_1)-1)\bar{d_1}\sqrt{2} & \phi(d_2)a_{12} & \cdots & \phi(d_t)a_{1t} & 0 \\
\sqrt{\alpha^2+ \bar{d_2}^2} & \phi(n)\sqrt{\beta^2+ \bar{d_2}^2} & \phi(d_1)a_{21} & (\phi(d_2)-1)\bar{d_2}\sqrt{2} & \cdots & \phi(d_t)a_{2t} & 0 \\
\vdots & \vdots & \vdots & \vdots & \ddots & \vdots & \vdots\\
\sqrt{\alpha^2+ \bar{d_t}^2} & \phi(n)\sqrt{\beta^2+\overline{d_t}^2} & \phi(d_1)a_{t1} & \phi(d_2)a_{t2} & \cdots & (\phi(d_t)-1)\bar{d_t}\sqrt{2} & 0 \\
\sqrt{\alpha^2+1^2} & 0 & 0 & 0 & \cdots & 0 & 0 \\
\end{bmatrix}
,\]
  where $\alpha=(2n-1), \; \beta=(n-1)$, $\bar{d_i}= d(a^j)$ for some $a^j$ with $o(a^j)= d_i$, and $a_{ij}=
\begin{dcases}
 \sqrt{\bar{d_i}^2+\bar{d_j}^2}, & d_i \mid d_j; \\
 0, & otherwise.
\end{dcases}$
\end{proof}

\begin{theorem}\label{power Q_2n}
 Let $G$ be a generalized quaternion group $Q_{4n}$. Then  $-3\sqrt{2}$ is a Sombor eigenvalue of multiplicity at  $n$ of power graph $\mathcal{P}(Q_{4n})$.
\end{theorem}

\begin{proof}
In view of Equation \ref{eq(2)}, we consider the sets $A = \{ a^i : 1 \leq i \leq 2n -1,~ i \neq n ~ \text{and} \; o(a^i)~ \text{is~ even}\}$ and $B = \{ a^i : 1 \leq i \leq 2n -1,~ i \neq n ~ \text{and} ~o(a^i)~ \text{is~ odd}\}$. Again, by Equation \ref{eq(2)}, we observe that $N[a^ib] = \{e, a^n, a^ib, a^{n+i}b\},$ where $1 \leq i \leq 2n$; $ ~ N[a^n] = A \cup \{ e \} \cup \{a^ib : 1 \leq i \leq n\}$ and the subgraph induced by the vertices belong to $\langle a \rangle$ is isomorphic to $\mathcal{P}(\mathbb Z_{2n}) =  K_{\phi(2n) + 1} \vee \Gamma_{n}[K_{\phi(d_1)}, K_{\phi(d_2)},\ldots,K_{\phi(d_t)}]$. Therefore the power graph $\mathcal{P}(Q_{4n})$ is given in Figure \ref{Fig-1}.

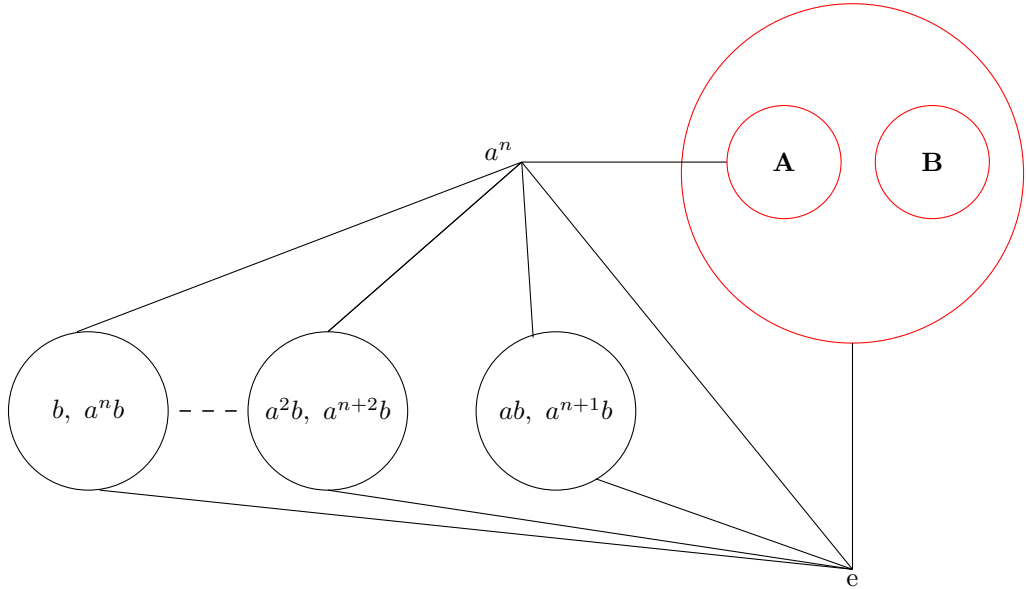
\begin{figure}[h!]
\begin{center}
\begin{tikzpicture}[scale=1.5]
\draw[color=red] (0,2) circle [radius=1.5];
\draw (0,-1.6) node{e};

\draw (-0.6,2.1) node{\textbf{A}};
\draw (0.7,2.1) node{\textbf{B}};
\draw (-1.1,2.1)--(-2.9,2.1);
\draw (-3.1,2.2) node{$a^n$};
\draw (-2.9,2.1)--(0,-1.5);
\draw (0,-1.5)--(0,0.5);
\draw[color=red] (-0.6, 2.1) circle [radius=0.5];
\draw[color=black] (-2.6, -0.1) circle [radius=0.7];
\draw (-2.9,2.1)--(-2.8,0.55);
\draw (-2.25, -0.7)--(0,-1.5);
\draw (-2.6, -0.1) node{$ab,~a^{n+1}b$};

\draw[color=black] (-4.6, -0.1) circle [radius=0.7];
\draw (-4.6, -0.8)--(0,-1.5);
\draw (-4.6, 0.6)--(-2.9,2.1);
\draw  (-5.6, -0.1)--(-5.7,-0.1);
\draw (-5.5, -0.1)--(-5.4,-0.1);
\draw (-5.8, -0.1)--(-5.9,-0.1);

\draw[color=black] (-6.7, -0.1) circle [radius=0.7];
\draw (-6.7, -0.1) node{$b,~a^{n}b$};
\draw (-2.9,2.1)--(-6.8,0.6);
\draw (-6.6, -0.8)--(0,-1.5);
\draw (-4.6, 0.6)--(-2.9,2.1);
\draw  (-5.6, -0.1)--(-5.7,-0.1);
\draw (-5.5, -0.1)--(-5.4,-0.1);
\draw (-5.8, -0.1)--(-5.9,-0.1);

\draw (-4.6, -0.1) node{$a^2b,~a^{n+2}b$};

\draw[color=red] (0.7,2.1) circle [radius=0.5];
\end{tikzpicture}
\caption{Power graph $\mathcal{P}(Q_{4n})$}\label{Fig-1}
\end{center}
\end{figure}

For each $1 \leq i \leq n$, note that $N[a^ib] = N[a^{n+i}b] \setminus \{a^ib, a^{n+i}b\} = \{e, a^n\}$ and $d(a^ib) = d(a^{n+i}b) = 3$.  By Lemma \ref{multiplicity}, $-3\sqrt{2}$ is an eigenvalue of multiplicity at least $n$. 
\end{proof}

\begin{theorem}\label{power SD_2n}
 Let $G$ be a semidihedral  group $SD_{8n}$. Then  $0$ and $-3\sqrt{2}$  are Sombor eigenvalue of a power graph $\mathcal{P}(SD_{8n})$ with multiplicity $n$ and $2n-1$, respectively.
\end{theorem}

\begin{proof}
In view of Equation \ref{Eq-3}, we consider the sets $A = \{ a^i : 1 \leq i \leq 2n -1,~ i \neq 2n ~ \text{and}~ o(a^i)~ \text{is~ even}\}$ and $B = \{ a^i : 1 \leq i \leq 2n -1,~ i \neq 2n ~ \text{and}~ o(a^i)~ \text{is~ odd}\}$. Again by Equation \ref{eq(2)}, we observe that $N[a^{2i +1}b] = \{e, a^{2n}, a^{2i + 1}b, a^{2n + 2i + 1}b\},$ where $1 \leq i \leq n$; $N[a^{2i}b] = \{e, a^{2i}b\},$ where $1 \leq i \leq n$; $ ~ N[a^{2n}] = A \cup \{ e \} \cup \{a^{2i+1}b : 1 \leq i \leq n\}$ and the subgraph induced by the vertices belong to $\langle a \rangle$ is isomorphic to $\mathcal{P}(\mathbb Z_{4n}) =  K_{\phi(4n) + 1} \vee \Gamma_{n}[K_{\phi(d_1)}, K_{\phi(d_2)},\ldots,K_{\phi(d_t)}]$. Therefore, the graph $\mathcal{P}(SD_{8n})$ is shown in Figure \ref{Fig-2}.

\begin{figure}[h!]
\begin{center}
\begin{tikzpicture}[scale=1.5]
\draw[color=red] (0,2) circle [radius=1.5];
\draw (0,-1.6) node{e};

\draw (-0.6,2.1) node{\textbf{A}};
\draw (0.7,2.1) node{\textbf{B}};
\draw (-1.1,2.1)--(-2.9,2.1);
\draw (-3.1,2.2) node{$a^{2n}$};
\draw (-2.9,2.1)--(0,-1.5);
\draw (0,-1.5)--(0,0.5);

\draw (0,-1.5)--(1.5,0);
\draw (1.7,0) node{\textbf{$a^2b$}};

\draw (0,-1.5)--(1.5,-0.4);
\draw (1.7,-0.4) node{\textbf{$a^4b$}};

\draw (1.7,-0.8)--(1.7,-0.9);
\draw (1.7,-1)--(1.7,-1.1);
\draw (1.7,-1.2)--(1.7,-1.3);

\draw (0,-1.5)--(1.5,-1.6);
\draw (1.7,-1.6) node{\textbf{$b$}};
\draw[color=red] (-0.6, 2.1) circle [radius=0.5];
\draw[color=black] (-2.6, -0.1) circle [radius=0.7];
\draw (-2.9,2.1)--(-2.8,0.55);
\draw (-2.25, -0.7)--(0,-1.5);
\draw (-2.6, -0.1) node{$ab,~a^{2n+1}b$};

\draw[color=black] (-4.6, -0.1) circle [radius=0.7];
\draw (-4.6, -0.8)--(0,-1.5);
\draw (-4.6, 0.6)--(-2.9,2.1);
\draw  (-5.6, -0.1)--(-5.7,-0.1);
\draw (-5.5, -0.1)--(-5.4,-0.1);
\draw (-5.8, -0.1)--(-5.9,-0.1);

\draw[color=black] (-6.7, -0.1) circle [radius=0.7];
\draw (-6.7, -0.1) node{$a^{2n-1}b$,};
\draw (-6.7, -0.5) node{$a^{4n-1}b$};
\draw (-2.9,2.1)--(-6.8,0.6);
\draw (-6.6, -0.8)--(0,-1.5);
\draw (-4.6, 0.6)--(-2.9,2.1);
\draw  (-5.6, -0.1)--(-5.7,-0.1);
\draw (-5.5, -0.1)--(-5.4,-0.1);
\draw (-5.8, -0.1)--(-5.9,-0.1);

%
\draw (-4.6, -0.1) node{$a^3b,~a^{2n+3}b$};

\draw[color=red] (0.7,2.1) circle [radius=0.5];
\end{tikzpicture}
\caption{Power graph $\mathcal{P}(SD_{8n})$}\label{Fig-2}
\end{center}
\end{figure}
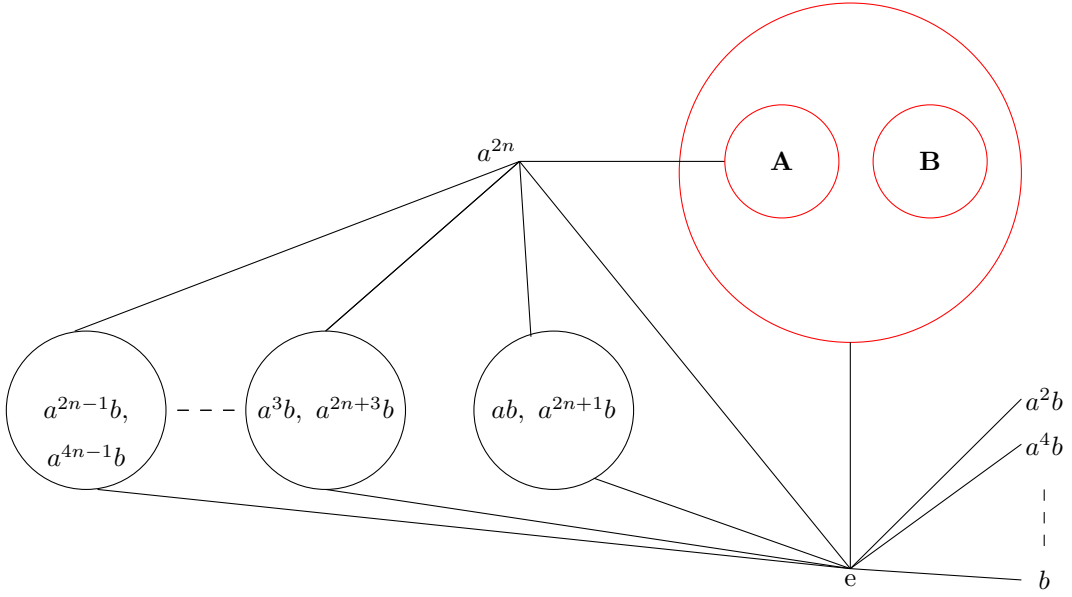

Let $S = \{ a^{2i}b : 1 \leq i \leq n \}$ be a subset of a vertex set $SD_{8n}$. In view of Figure \ref{Fig-2}, the vertices belonging to the set $S$ are independent and $N[x] \setminus S = \{e\}$ for all $x \in S$.  By Lemma \ref{multiplicity}, $0$ is an eigenvalue of multiplicity at least $2n-1$. Moreover, $-3 \sqrt{2}$ is a Sombor eigenvalue of $\mathcal{P}(SD_{8n})$ with multiplicity $2n-1$.

\end{proof}

\subsection{Order Super power Graph}
In this subsection, we obtain the Sombor spectrum of the order super power graph $\mathcal{P}^o(G)$, for the group $G= D_{2n}, Q_{4n}$ and $SD_{8n}.$

\begin{theorem}\label{orderSP D_2n}
  Let $\mathcal{P}^o(D_{2n})$ be the order super power graph of the dihedral group. 

\noindent(i) For odd $n$,  the Sombor spectrum of order super power graph $\mathcal{P}^o(D_{2n})$ is
 \[\displaystyle \begin{pmatrix}
-(n-1)\sqrt{2} & -\bar{d}_1\sqrt{2}  & -\bar{d}_2\sqrt{2} & \cdots & -\bar{d}_t\sqrt{2} & 0 \\
 \phi(n)-1 & \phi(d_1)-1 & \phi(d_2)-1 & & \phi(d_t)-1 & n-1  \\
\end{pmatrix},
\]
and the remaining eigenvalues are the eigenvalues of the following equitable quotient matrix 
\[\begin{bmatrix}
0 & \phi(n)\sqrt{\alpha^2+\beta^2} & \phi(d_1)\sqrt{\alpha^2+ \bar{d_1}^2}  & \phi(d_2)\sqrt{\alpha^2+\bar{d_2}^2} & \cdots & \phi(t) \sqrt{\alpha^2+ \bar{d_t}^2} & n\sqrt{\alpha^2+n^2} \\
\sqrt{\alpha^2+\beta^2} & (\phi(n)-1)\beta\sqrt{2} & \phi(d_1)\sqrt{\beta^2+ \bar{d_1}^2} & \phi(d_2)\sqrt{\beta^2+\bar{d_2}^2} & \cdots & \phi(d_t)\sqrt{\beta^2+ \bar{d_t}^2} & 0\\
\sqrt{\alpha^2+ \bar{d_1}^2} & \phi(n)\sqrt{\beta^2+ \bar{d_1}^2} & (\phi(d_1)-1)\bar{d_1}\sqrt{2} & \phi(d_2)a_{12} & \cdots & \phi(d_t)a_{1t} & 0 \\
\sqrt{\alpha^2+\bar{d_2}^2} & \phi(n)\sqrt{\beta^2+\bar{d_2}^2} & \phi(d_1)a_{21} &  (\phi(d_2)-1)\bar{d_2}\sqrt{2} & \cdots & \phi(d_t)a_{2t} & 0 \\
\vdots & \vdots & \vdots & \vdots & \ddots & \vdots & \vdots\\
\sqrt{\alpha^2+\bar{d_t}^2} & \phi(n)\sqrt{\beta^2+\bar{d_t}^2} &\phi(d_1)a _{t1} & \phi(d_2)a_{t2} & \cdots & \bar{d_t}\sqrt{2}(\phi(d_t)-1)& 0\\
\sqrt{\alpha^2+n^2} & 0 & 0 & 0 & \cdots & 0 &  (n-1)n\sqrt{2} \\
\end{bmatrix},\]
where $\alpha= 2n-1, \beta=n-1$, $\bar{d_i}= d(a^j)$ for some $a^j$ with $o(a^j)= d_i$, and $a_{ij}=
\begin{dcases}
 \sqrt{\bar{d_i}^2+\bar{d_j}^2}, & d_i \mid d_j; \\
 0, & otherwise.
\end{dcases}$

\noindent(ii) For even $n$, the Sombor spectrum of order super power graph $\mathcal{P}^o(D_{2n})$ is 
\[ \begin{pmatrix}
-(2n-1)\sqrt{2} & -\bar{d}_1\sqrt{2}  & -\bar{d}_2\sqrt{2} & \cdots & -\bar{d}_t\sqrt{2} & -\bar{d}_0\sqrt{2} \\
 \phi(n) & \phi(d_1)-1 & \phi(d_2)-1 & & \phi(d_t)-1 & n  \\
\end{pmatrix}
\]
and the remaining eigenvalues are the eigenvalues of the following equitable quotient matrix
\[
\begin{bmatrix}
 \phi(n)\alpha\sqrt{2} & \phi(d_1)\sqrt{\alpha^2+\bar{d_1}^2}  & \phi(d_2)\sqrt{\alpha^2+\bar{d_2}^2}  & \cdots & \phi(d_t)\sqrt{\alpha^2+\bar{d_t}^2} &  (n+1)\sqrt{\alpha^2+\bar{d}_o^2} \\
(\phi(n)+1)\sqrt{\alpha^2+\bar{d_1}^2} & (\phi(d_1)-1)\bar{d_1}\sqrt{2} & \phi(d_2)a_{12} & \cdots & \phi(d_t)a_{1t} & (n+1)a_{10}\\
((\phi(n)+1))\sqrt{\alpha^2+\bar{d_2}^2} & \phi(d_1)a_{21} &  (\phi(d_2)-1)\bar{d_2}\sqrt{2} & \cdots & \phi(d_t)a_{2t} & (n+1)a_{20} \\
\vdots & \vdots & \vdots & \ddots & \vdots & \vdots\\
((\phi(n)+1))\sqrt{\alpha^2+\bar{d_t}^2} & \phi(d_1)a_{t1} & \phi(d_2)_{t2} & \cdots & (\phi(d_t)-1)\bar{d_t}\sqrt{2} & (n+1)a_{t0} \\
(\phi(n)+1)\sqrt{\alpha^2+\bar{d}_o^2}  & \phi(d_1)a_{01} & \phi(d_2)a_{02} & \cdots & \phi(d_t)a_{0t} & n\bar{d_0}\sqrt{2},\end{bmatrix}\]
where $\alpha= 2n-1$, $\bar{d_i}= d(a^j)$ for some $a^j$ with $o(a^j)= d_i$, and $a_{ij}=
\begin{dcases}
 \sqrt{\bar{d_i}^2+\bar{d_j}^2}, & d_i \mid d_j; \\
 0, & otherwise.
\end{dcases}$
\end{theorem}

\begin{proof}\noindent(i) First we explore the structure of $\mathcal{P}^o(D_{2n})$. In view of the representation of $D_{2n}$, we have $o(a^ib) = 2$ for all $i$, where $1 \leq i \leq n$. It follows that the subgraph induced by the vertices of the form $a^ib$ is complete. Further, note that $\langle a^ib \rangle = \{e, a^ib\}$ for all $i$ gives $a^i \nsim a^jb$ for all $1 \leq i, j \leq n$ in $\mathcal{P}(D_{2n})$. Since $n$ is odd so that $a^i \nsim a^jb$ for all $1 \leq i, j \leq n$ in $\mathcal{P}^o(D_{2n})$. Also,  $x \sim y$ in $\mathcal{P}(\mathbb Z_n)$ if and only if $x \sim y$ in $\mathcal{P}^o(\mathbb Z_n)$ (see Theorem \ref{P(G)=P^o(G)}). Therefore, we have $\mathcal{P}^o(\mathbb Z_n) = \mathcal{P}(\mathbb Z_n)$ which is isomorphic to the subgraph of $\mathcal{P}^o(D_{2n})$ induced by the vertices belongs to $\langle a \rangle$. By using representation of $\mathcal{P}(\mathbb Z_n)$ given in the proof of Lemma \ref{P(D2n)-structure}, we have
\[ 
\mathcal{P}^o(D_{2n}) = K_1 \vee \Big( K_{\phi(n)} \vee \Gamma_{n}[K_{\phi(d_1)},K_{\phi(d_2)},\ldots,K_{\phi(d_t)}] \cup {K}_n \Big),
\]
where $\Gamma_n$ is a graph with vertex $V(\Gamma_n) = \{ d_i : 1, n \neq d_i | n,~ 1 \leq i \leq t \}$ and two distinct  vertices $d_i$ and $d_j$ are adjacent in $\Gamma_n$ if one of them divides other. By using the same argument used in Theorem \ref{power D_2n}, we have the result. 
\noindent(ii) As $n$ is even, $o(a^{\frac{n}{2}}) = 2$. Therefore, $N[a^{\frac{n}{2}}] = N[a^ib]$ for all $i$, where $1 \leq i \leq n$. Thus, we have  
\[ 
\mathcal{P}^o(D_{2n}) = K_{1+\phi(n)} \vee \Gamma_{n}[ K_{n+1},K_{\phi(d_1)},K_{\phi(d_2)},\ldots,K_{\phi(d_t)}],
\]
where $\Gamma_n$ is a graph with vertex $V(\Gamma_n) = \{ d_i : 1, n \neq d_i | n,~ 1 \leq i \leq t, i \neq 2 \}$ and two distinct  vertices $d_i$ and $d_j$ are adjacent in $\Gamma_n$ if one of them divides other. Again, by using the same argument used in Theorem \ref{power D_2n}, we have the result. 
\end{proof}

\begin{theorem}\label{power Q_4n}
 Let $G$ be a generalized quaternion group $Q_{4n}$. Then we have the following.
\begin{enumerate}
\item[\rm (i)] For even $n$, $-\sqrt{2} (2n + 3 + |C|)$ is a Sombor eigenvalues of $\mathcal{P}^o(Q_{4n})$  with multiplicity at least $2n + 1$, where $C = \{ a^i : 1 \leq i \leq 2n -1,~ \text{and}~ 4\mid o(a^i) \}$ {(\rm see Equation \ref{Eq-3})}. 

\item[\rm (ii)] For odd $n$, $-\sqrt{2}(2n + 1)$ is a Sombor eigenvalue of $\mathcal{P}^o(Q_{4n})$ with multiplicity of at least $2n-1$. 
\end{enumerate} 
\end{theorem}

\begin{proof}(i) In view of Equation \ref{eq(2)}, we consider the sets $A = \{ a^i : 1 \leq i \leq 2n -1,~ i \neq n, ~ 2 |o(a^i) ~ \text{and}~ 4 \nmid  o(a^i)\}$, $B = \{ a^i : 1 \leq i \leq 2n -1 ~ {\rm and}~o(a^i)~ \text{is~ odd}\}$, $C = \{ a^i : 1 \leq i \leq 2n -1,~ \text{and}~ 4\mid o(a^i) \}$ and $D = \{a^ib : ~1 \leq i \leq 4n \}$$ \cup \{a^{\frac{n}{2}},a^{\frac{3n}{2}}\}$. Again, by Equation \ref{eq(2)}, we observe that $o(a^{\frac{n}{2}}) = o(a^{\frac{3n}{2}}) = o(a^ib) = 4$ for all $i$, where $1 \leq i \leq 2n$. This implies that the subgraph induced by the vertices belongs to the set $D$ forms a clique and $N[x] = N[y]$ for all $x, y \in D$. In addition, we note that $a^ib \nsim a^j$ for all $1 \leq i,j \leq 2n$ and $j \notin \{\frac{n}{2}, \frac{3n}{2}, n, 2n\}$ in $\mathcal{P}^o(Q_{4n})$. Consequently, we get $N[x] = C \cup D \cup \{e, a^{n}\}$ for all $x \in D$, $N[a^n] =  A \cup C \cup D \cup  \{e, a^n \}$ and the subgraph induced by the vertices belongs to $\langle a \rangle$ is isomorphic to $\mathcal{P}(\mathbb Z_{2n}) =  K_{\phi(2n) + 1} \vee \Gamma_{2n}[K_{\phi(d_1)}, K_{\phi(d_2)},\ldots,K_{\phi(d_t)}]$. Also, $a^n \nsim x$ for all $x \in B$. Therefore, we obtain the order super power graph $\mathcal{P}(Q_{4n})$ which is given in Figure \ref{Fig-3}. Thus,  $-\sqrt{2} (2n + 3 + |C|)$ is a Sombor eigenvalue of $\mathcal{P}^o(Q_{4n})$ with multiplicity at least $2n + 1$.

\begin{figure}[h!]
\begin{center}
\begin{tikzpicture}[scale=1.5]
\draw[color=red] (0,2) circle [radius=1.5];
\draw (0,-1.6) node{e};

\draw (-0.6,2.1) node{\textbf{A}};
\draw (0.7,2.1) node{\textbf{B}};
\draw (-1.1,2.1)--(-2.9,2.1);
\draw (-2.8,2.3) node{$a^n$};
\draw (-2.9,2.1)--(0,-1.5);
\draw (0,-1.5)--(0,0.5);
\draw[color= red] (-0.6, 2.1) circle [radius=0.5];
\draw[color=black] (-2.6, -0.1) circle [radius=0.7];
\draw (-2.9,2.1)--(-2.8,0.55); 
\draw (-2.25, -0.7)--(0,-1.5);
\draw (-2.6, -0.1) node{$D$};
\draw (-0.3,1.4)--(-2.9,2.1); 



%
\draw (-2.8,0.55)--(-0.3,1.4); 
\draw[color= red] (0.1,1.1) circle [radius=0.5]; 
\draw (0.1,1.1) node{\textbf{C}};
\draw[color= red] (0.7,2.1) circle [radius=0.5];
\end{tikzpicture}
\caption{Power graph $\mathcal{P}^o(Q_{4n})$, where $n$ is even}\label{Fig-3}
\end{center}
\end{figure}
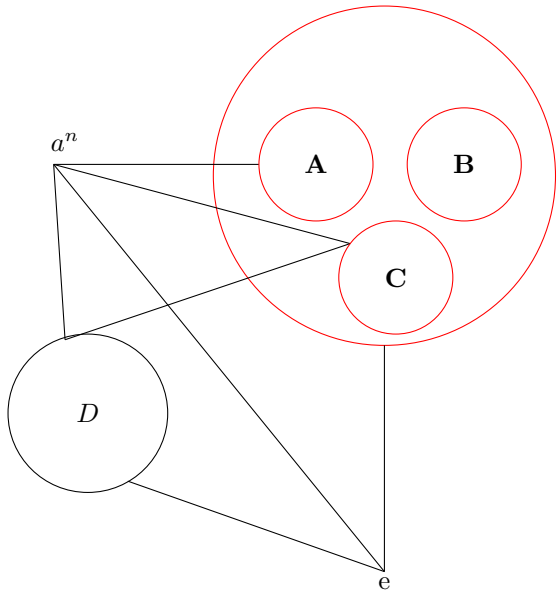
 \noindent (ii) As $n$ is odd, it implies that $ \langle a \rangle$ has no element of order $4$. 
 Now, consider the sets $A = \{ a^i : 1 \leq i \leq 2n -1,~ i \neq n \; {\rm and}~o(a^i)~ \text{is~ even}\}$, $B = \{ a^i : 1 \leq i \leq 2n -1 ~ {\rm and}~o(a^i)~ \text{is~ odd}\}$ and $D = \{a^ib : ~1 \leq i \leq 2n \}$. Again, by Equation \ref{eq(2)}, we observe that $N[x] = \{e, a^n\} \cup D$ for all $x \in D$, $N[a^n] = \{e\} \cup D \cup A$ and the subgraph induced by the vertices belong to $\langle a \rangle$ is isomorphic to $\mathcal{P}(\mathbb Z_{2n}) =  K_{\phi(2n) + 1} \vee \Gamma_{2n}[K_{\phi(d_1)}, K_{\phi(d_2)},\ldots,K_{\phi(d_t)}]$. Therefore, the order super power graph $\mathcal{P}(Q_{4n})$ is given in Figure \ref{Fig-4}.

\begin{figure}[h!]
\begin{center}
\begin{tikzpicture}[scale=1.5]
\draw[color=red] (0,2) circle [radius=1.5];
\draw (0,-1.6) node{e};

\draw (-0.6,2.1) node{\textbf{A}};
\draw (0.7,2.1) node{\textbf{B}};
\draw (-1.1,2.1)--(-2.9,2.1);
\draw (-2.8,2.3) node{$a^{n}$};
\draw (-2.9,2.1)--(0,-1.5);
\draw (0,-1.5)--(0,0.5);
\draw[color=red] (-0.6, 2.1) circle [radius=0.5];
\draw[color=black] (-2.6, -0.1) circle [radius=0.7];
\draw (-2.9,2.1)--(-2.8,0.55); 
\draw (-2.25, -0.7)--(0,-1.5);
\draw (-2.6, -0.1) node{$D$};



%

\draw[color=red] (0.7,2.1) circle [radius=0.5];
\end{tikzpicture}






\caption{Power graph $\mathcal{P}^o(Q_{4n})$, where $n$ is odd}\label{Fig-4}
\end{center}
\end{figure}
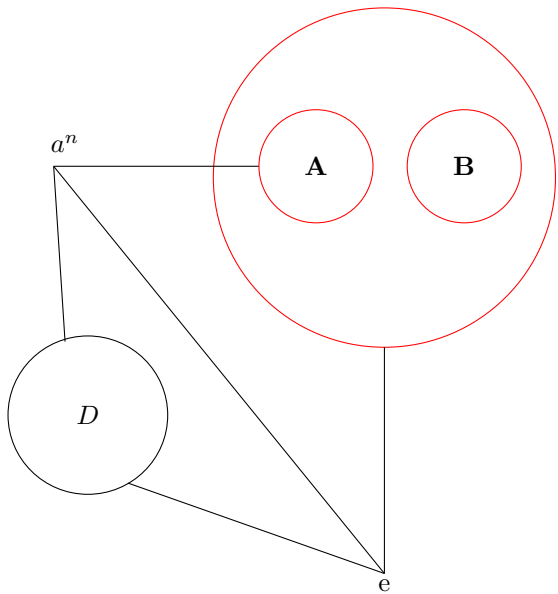
 In view of Figure \ref{Fig-4}, the vertices belonging to the set $D$ form a clique of size $2n$ and $N[x] \setminus S = \{e, a^n\}$ for all $x \in D$.  By Theorem \ref{multiplicity}, $-\sqrt{2}(2n + 1)$ is an eigenvalue of multiplicity of at least $2n-1$. 
\end{proof}

\begin{theorem}\label{power SD_8n}
 Let $G$ be a semidihedral  group $SD_{8n}$. Then  $-\sqrt{2} (4n + 3 + |A| + |C|)$ and $-\sqrt{2} (4n + 3 + |C|)$ are the Sombor eigenvalues of $\mathcal{P}(SD_{8n})$ with multiplicity at least $2n$ and $2n + 1$, respectively, where $C = \{ a^i : 1 \leq i \leq 2n -1,~ \text{and}~ 4\mid o(a^i) \}$ {\rm (see Equation \ref{Eq-3})}. 


\end{theorem}

\begin{proof} In view of Equation \ref{eq(2)}, we consider the sets $A = \{ a^i : 1 \leq i \leq 4n -1,~ i \neq 2n, ~ 2 |o(a^i) ~ \text{and}~ 4 \nmid  o(a^i)\}$, $B = \{ a^i : 1 \leq i \leq 4n -1 ~ {\rm and}~o(a^i)~ \text{is~ odd}\}$, $C = \{ a^i : 1 \leq i \leq 2n -1,~ \text{and}~ 4\mid o(a^i) \} \setminus \{a^n, a^{3n}\}$, $D = \{a^{2i+1}b : ~1 \leq i \leq n \} \cup \{a^n, a^{3n}\}$ and $E = \{a^{2i}b : ~1 \leq i \leq n \} \cup \{a^{2n}\}$ . Again, by Equation \ref{eq(2)}, we observe that $o(a^{n}) = o(a^{3n}) = o(a^{2i+1}b) = 4$ for all $i$, where $1 \leq i \leq 2n$ and $o(a^{2n}) = o(a^{2j}b) = 2$ for all $j$, where $1 \leq j \leq 2n$ . This implies that the subgraph induced by the vertices belongs to the set $D$ and $E$ form cliques of sizes $2n+2$ and $2n+1$, respectively. In addition, we note that $a^ib \nsim a^j$ for all $1 \leq i,j \leq 2n$ and $i \notin \{n, 2n, 3n, 4n\}$ in $\mathcal{P}(SD_{8n})$. Consequently, we get  $N[x] =C \cup D \cup E \cup \{e\}$ for all $x\in D$ and $N[y] = C \cup D\cup E \cup A \cup \{e\}$ for all $y \in E$ and the subgraph induced by the vertices belongs to $\langle a \rangle$ is isomorphic to $\mathcal{P}(\mathbb Z_{4n}) =  K_{\phi(4n) + 1} \vee \Gamma_{4n}[K_{\phi(d_1)}, K_{\phi(d_2)},\ldots,K_{\phi(d_t)}]$. Therefore, the order super power graph $\mathcal{P}(SD_{8n})$ is given in Figure \ref{Fig-5}. Thus, by Lemma \ref{multiplicity},  $-\sqrt{2} (4n + 3 + |A| + |C|)$, and $-\sqrt{2} (4n + 3 + |C|)$ are the Sombor eigenvalues of $\mathcal{P}(SD_{8n})$ with multiplicity at least $2n$ and $2n + 1$, respectively.
\begin{figure}[h!]
\begin{center}
\begin{tikzpicture}[scale=1.5]
\draw[color=red] (0,2) circle [radius=1.5];
\draw (0,-1.6) node{e};

\draw (-0.6,2.1) node{\textbf{A}};
\draw (0.7,2.1) node{\textbf{B}};
\draw (-1.1,2.1)--(-2.9,2.1);
\draw (-3.4,2.6) node{E};
\draw (-2.9,2.1)--(0,-1.5);
\draw (0,-1.5)--(0,0.5);
\draw[color=red] (-0.6, 2.1) circle [radius=0.5];
\draw[color=black] (-2.6, -0.1) circle [radius=0.7];
\draw (-2.9,2.1)--(-2.8,0.55); 
\draw (-2.25, -0.7)--(0,-1.5);
\draw (-2.6, -0.1) node{$D$};
\draw (-0.3,1.4)--(-2.9,2.1); 


\draw[color=red] (-0.6, 2.1) circle [radius=0.5];
\draw[color=black] (-3.4, 2.6) circle [radius=0.7];
\draw (-2.9,2.1)--(-2.8,0.55); 
\draw (-2.25, -0.7)--(0,-1.5);
\draw (-2.6, -0.1) node{$D$};
\draw (-0.3,1.4)--(-2.9,2.1); 



%
\draw (-2.8,0.55)--(-0.3,1.4); 
\draw[color=red] (0.1,1.1) circle [radius=0.5]; 
\draw (0.1,1.1) node{\textbf{C}};
\draw[color=red] (0.7,2.1) circle [radius=0.5];
\end{tikzpicture}
\caption{Power graph $\mathcal{P}^o(SD_{8n})$ }\label{Fig-5}
\end{center}
\end{figure}
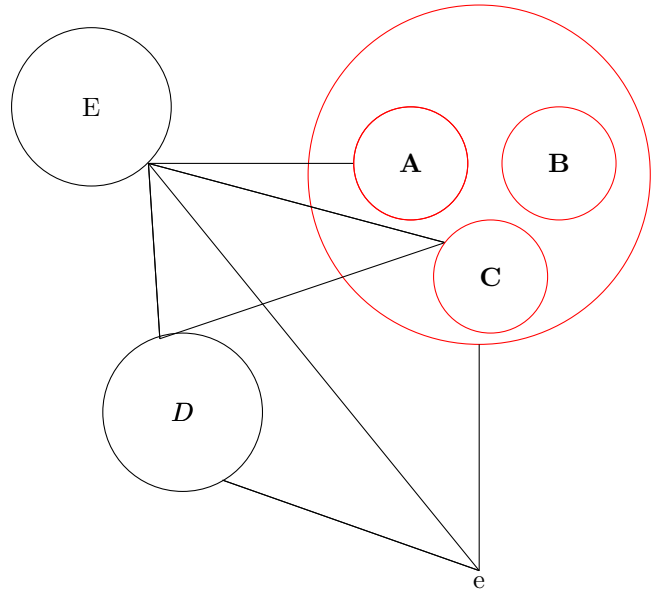

\end{proof}

\subsection{Conjugacy super power graph}
In this subsection, we  discuss the Sombor spectrum of conjugacy super power graph $\mathcal{P}^c(D_{2n})$, for $G= D_{2n}, Q_{4n}$ and $SD_{8n}.$
 
\begin{theorem}\label{conjugacySP D_2n}
 Let $\mathcal{P}^c(D_{2n})$ be the conjugacy superpower graph of the dihedral group. Then we have the following statements.
\begin{enumerate}
\item[\rm (i)] If $n$ is odd, then the Sombor spectrum is
 \[\displaystyle \begin{pmatrix}
-(n-1)\sqrt{2} & -\bar{d}_1\sqrt{2}  & -\bar{d}_2\sqrt{2} & \cdots & -\bar{d}_t\sqrt{2} & 0 \\
 \phi(n)-1 & \phi(d_1)-1 & \phi(d_2)-1 & & \phi(d_t)-1 & n-1  \\
\end{pmatrix},
\]
and the remaining eigenvalues are the roots of the following equitable quotient matrix 
\[\begin{bmatrix}
0 & \phi(n)\sqrt{\alpha^2+\beta^2} & \phi(d_1)\sqrt{\alpha^2+\bar{d_1}^2}  & \phi(d_2)\sqrt{\alpha^2+\bar{d_2}^2} & \cdots & \phi(t) \sqrt{\alpha^2+\bar{d_t}^2} & n\sqrt{\alpha^2+n^2} \\
\sqrt{\alpha^2+\beta^2} & (\phi(n)-1)\beta\sqrt{2} & \phi(d_1)\sqrt{\beta^2+\bar{d_1}^2} & \phi(d_2)\sqrt{\beta^2+\bar{d_2}^2} & \cdots & \phi(d_t)\sqrt{\beta^2+\bar{d_t}^2} & 0\\
\sqrt{\alpha^2+\bar{d_1}^2} & \phi(n)\sqrt{\beta^2+\bar{d_1}^2} & (\phi(d_1)-1)\bar{d_1}\sqrt{2} & \phi(d_2)a_{12} & \cdots & \phi(d_t)a_{1t} & 0 \\
\sqrt{\alpha^2+\bar{d_2}^2} & \phi(n)\sqrt{\beta^2+\bar{d_2}^2} & \phi(d_1)a_{21} &  (\phi(d_2)-1)\bar{d_2}\sqrt{2} & \cdots & \phi(d_t)a_{2t} & 0 \\
\vdots & \vdots & \vdots & \vdots & \ddots & \vdots & \vdots\\
\sqrt{\alpha^2+\bar{d_t}^2} & \phi(n)\sqrt{\beta^2+\bar{d_t}^2} &\phi(d_1)a _{t1} & \phi(d_2)a_{t2} & \cdots & \bar{d_t}\sqrt{2}(\phi(d_t)-1)& 0\\
\sqrt{\alpha^2+n^2} & 0 & 0 & 0 & \cdots & 0 &  (n-1)n\sqrt{2} 
\end{bmatrix},\]
where $\alpha= 2n-1, ~ \beta=n-1$, ~ $\bar{d_i}= d(a^j)$ for some $a^j$ with $o(a^j)= d_i$, and $a_{ij}=
\begin{dcases}
 \sqrt{\bar{d_i}^2+\bar{d_j}^2}, & d_i \mid d_j; \\
 0, & otherwise.
\end{dcases}$.

\item[\rm (ii)] If $n$ is even, then the Sombor spectrum is 

 \[\displaystyle \begin{pmatrix}
-(n-1)\sqrt{2} & -\bar{d}_1\sqrt{2}  & -\bar{d}_2\sqrt{2} & \cdots & -\bar{d}_t\sqrt{2} & -\frac{n}{2}\sqrt{2} \\
 \phi(n)-1 & \phi(d_1)-1 & \phi(d_2)-1 & & \phi(d_t)-1 & n-2  \\
\end{pmatrix},
\]
and the remaining eigenvalues are the roots of the following equitable quotient matrix 
\[
\scalebox{0.9}{$
\begin{bmatrix}
0 & \phi(n)\sqrt{\alpha^2+\beta^2} & \phi(d_1)\sqrt{\alpha^2+\bar{d_1}^2}  & \phi(d_2)\sqrt{\alpha^2+\bar{d_2}^2} & \cdots & \phi(t) \sqrt{\alpha^2+\bar{d_t}^2} & \frac{n}{2}\sqrt{\alpha^2+(\frac{n}{2})^2} & \frac{n}{2}\sqrt{\alpha^2+(\frac{n}{2})^2} \\
\sqrt{\alpha^2+\beta^2} & (\phi(n)-1)\beta\sqrt{2} & \phi(d_1)\sqrt{\beta^2+\bar{d_1}^2} & \phi(d_2)\sqrt{\beta^2+\bar{d_2}^2} & \cdots & \phi(d_t)\sqrt{\beta^2+\bar{d_t}^2} & 0 & 0\\
\sqrt{\alpha^2+\bar{d_1}^2} & \phi(n)\sqrt{\beta^2+\bar{d_1}^2} & (\phi(d_1)-1)\bar{d_1}\sqrt{2} & \phi(d_2)a_{12} & \cdots & \phi(d_t)a_{1t} & 0 & 0 \\
\sqrt{\alpha^2+\bar{d_2}^2} & \phi(n)\sqrt{\beta^2+\bar{d_2}^2} & \phi(d_1)a_{21} &  (\phi(d_2)-1)\bar{d_2}\sqrt{2} & \cdots & \phi(d_t)a_{2t} & 0 & 0 \\
\vdots & \vdots & \vdots & \vdots & \ddots & \vdots & \vdots & \vdots\\
\sqrt{\alpha^2+\overline{d_t}^2} & \phi(n)\sqrt{\beta^2+\bar{d_t}^2} &\phi(d_1)a _{t1} & \phi(d_2)a_{t2} & \cdots & \bar{d_t}\sqrt{2}(\phi(d_t)-1) & 0 & 0 \\
\sqrt{\alpha^2+(\frac{n}{2})^2} & 0 & 0 & 0 & \cdots & 0 & \frac{n}{2}(\frac{n}{2}-1)\sqrt{2} & 0 \\
\sqrt{\alpha^2+(\frac{n}{2})^2} & 0 & 0 & 0  & \cdots & 0 & 0 & \frac{n}{2}(\frac{n}{2}-1)\sqrt{2} 
\end{bmatrix}
$}
\]
where $\alpha = 2n-1,~ \beta = n-1,~  \bar{d_i}= d(a^j)$ for some $a^j$ with $o(a^j)= d_i$, and $a_{ij}=
\begin{dcases}
 \sqrt{\bar{d_i}^2+\bar{d_j}^2}, & d_i \mid d_j; \\
 0, & otherwise.
\end{dcases}$
\end{enumerate}
\end{theorem}

\begin{proof}\noindent(i) For odd $n$, we claim $\mathcal{P}^c(D_{2n}) = \mathcal{P}^o(D_{2n}).$ As we know that the conjugacy classes of the dihedral group $D_{2n}$ are
\[ [e]_c = \{ e \}; ~ [a^i]_c = \{ a^i, a^{n-i}\},~{\rm for~all~} i,~{\rm where}~1 \leq i \leq \frac{n-1}{2};~ [b]_c = \{ a^ib : 1 \leq i \leq n \}.\]
Therefore, the subgraph induced by the vertices belongs to the conjugacy class $[b]_c$ is complete. Also, we have been already prove in Theorem \ref{orderSP D_2n} (i) that $a^ib \nsim a^j$ in $\mathcal{P}^o(D_{2n})$ for all $1 \leq i \leq n$ and $1 \leq j \leq n - 1$. It implies that $a^ib \nsim a^j$ in $\mathcal{P}^c(D_{2n})$ because the conjugacy class of $a^ib$ and $a^j$ are distinct. As a consequence, $N[a^ib]$ in both the graphs $\mathcal{P}^c(D_{2n})$ and $\mathcal{P}^o(D_{2n})$ are equal. Also, $N[e] = D_{2n}$ in both the graphs. Now we show that $x \sim y$ in $\mathcal{P}^o(D_{2n})$ if and only if $x \sim y$ in $\mathcal{P}^c(D_{2n})$ for all $x, y \in \langle a \rangle$. In view of Theorem \ref{P(G)=P^o(G)}, it is sufficient to prove that $x \sim y$ in $\mathcal{P}(D_{2n})$ if and only if $x \sim y$ in $\mathcal{P}^c(D_{2n})$ for all $x, y \in \langle a \rangle$. As we know that $\mathcal{P}(D_{2n})$ is a spanning subgraph of $\mathcal{P}^c(D_{2n})$. On the other side, let $x,y \in \langle a \rangle$ such that $x \sim y$ in $\mathcal{P}^c(D_{2n})$. Then either $x \sim y$ in $\mathcal{P}( D_{2n})$ or $x, y$ belongs to the same conjugacy class. For $x \sim y$ in $\mathcal{P}( D_{2n})$, nothing to prove it. In case of $x, y$ belongs to the same conjugacy class, we have $o(x) = o(y)$. Consequently, we get $x \sim y$ in $\mathcal{P}^o( D_{2n})$ and so $x \sim y$ in $\mathcal{P}( D_{2n})$ (see Theorem \ref{P(G)=P^o(G)}). Thus, we have $\mathcal{P}^c(D_{2n}) = \mathcal{P}^o(D_{2n})$ and  Sombor spectrum of $\mathcal{P}^o(D_{2n})$  has already been discussed in Theorem \ref{orderSP D_2n}.

\noindent(ii) For even $n$, the conjugacy classes of the dihedral group $D_{2n}$ are
$[e]_c = \{ e \};$  $[a^{\frac{n}{2}}]_c = \{ a^{\frac{n}{2}} \};  ~ [a^i]_c = \{ a^i, a^{n-i}\},~{\rm for~all~}\\ i,~{\rm where}~1 \leq i \leq \frac{n-2}{2};~[b]_c = \{ a^{2i}b : 1 \leq i \leq \frac{n}{2}\}$ and $[ab]_c = \{ a^{2i - 1}b : 1 \leq i \leq \frac{n}{2}\}$. The induced subgraphs of the graph $\mathcal{P}^c(D_{2n})$ with vertex sets $[b]_c$, and $[ab]_c$ are isomorphic to $K_{\frac{n}{2}}$ and $K_{\frac{n}{2}}$, respectively. Clearly, no vertex of the conjugacy class $[b]_c$ is adjacent to any element of the class $[ab]_c$. By using the same argument used in part (i), we get  
\[ 
\mathcal{P}^o(D_{2n}) = K_1 \vee \left[ K_{\phi(n)} \vee \Gamma_{n}[K_{\phi(d_1)},K_{\phi(d_2)},\ldots,K_{\phi(d_t)}] \cup {K}_{\frac{n}{2}} \cup {K}_{\frac{n}{2}}\right].
\]
where $\Gamma_n$ is a graph with vertex $V(\Gamma_n) = \{ d_i : 1, n \neq d_i | n,~ 1 \leq i \leq p \}$ and two distinct  vertices $d_i$ and $d_j$ are adjacent in $\Gamma_n$ if one of them divides other. By using the same argument used in Theorem \ref{power D_2n}, we have the result. 
\end{proof}

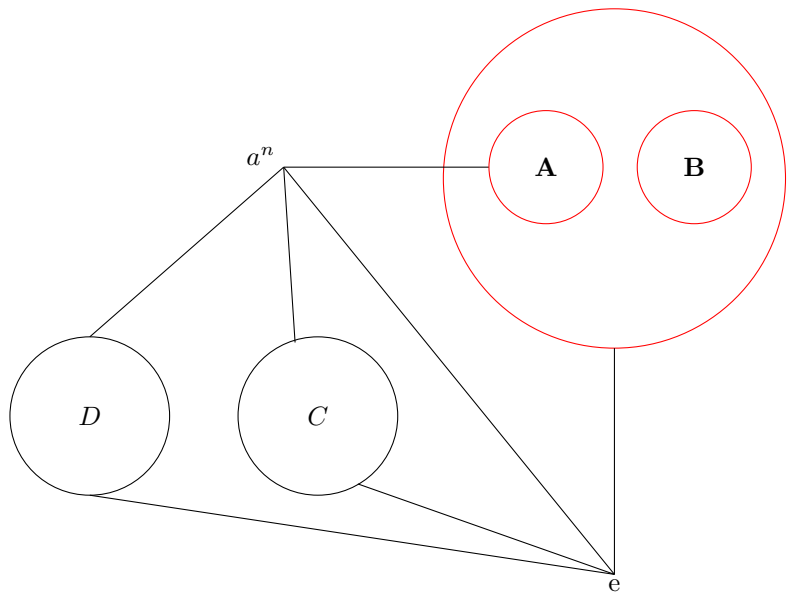
\begin{figure}[h!]
\begin{center}
\begin{tikzpicture}[scale=1.5]
\draw[color=red] (0,2) circle [radius=1.5];
\draw (0,-1.6) node{e};

\draw (-0.6,2.1) node{\textbf{A}};
\draw (0.7,2.1) node{\textbf{B}};
\draw (-1.1,2.1)--(-2.9,2.1);
\draw (-3.1,2.2) node{$a^n$};
\draw (-2.9,2.1)--(0,-1.5);
\draw (0,-1.5)--(0,0.5);



\draw[color=red] (-0.6, 2.1) circle [radius=0.5];
\draw[color=black] (-2.6, -0.1) circle [radius=0.7];
\draw (-2.9,2.1)--(-2.8,0.55);
\draw (-2.25, -0.7)--(0,-1.5);
\draw (-2.6, -0.1) node{$C$};

\draw[color=black] (-4.6, -0.1) circle [radius=0.7];
\draw (-4.6, -0.8)--(0,-1.5);
\draw (-4.6, 0.6)--(-2.9,2.1);
\draw (-4.6, -0.1) node{$D$};


%
\draw[color=red] (0.7,2.1) circle [radius=0.5];
\end{tikzpicture}
\caption{Power graph $\mathcal{P}^c(Q_{4n})$, when $n$ is even}\label{Fig-7}
\end{center}
\end{figure}

\begin{theorem}\label{power-conj-Q_4n}
 Let $G$ be a generalized quaternion group $Q_{4n}$.  Then we have the following.
 \begin{enumerate}
\item[\rm (i)] For even $n$, $-\sqrt{2}(n+1)$  is a Sombor eigenvalue of multiplicity of at least $2n-2$ of the power graph $\mathcal{P}^c(Q_{4n})$.

\item[\rm (ii)] For odd $n$, $-\sqrt{2}(2n+1)$ is a Sombor eigenvalue of multiplicity of at least $2n-1$ of the power graph $\mathcal{P}^c(Q_{4n})$.
 \end{enumerate}
 
\end{theorem}

\begin{proof} (i) In view of Equation \ref{eq(2)}, we consider the sets $A = \{ a^i : 1 \leq i \leq 2n -1,~ i \neq n ~ \text{and}~o(a^i)~ \text{is~ even}\}$, $B = \{ a^i : 1 \leq i \leq 2n -1,~ i \neq n ~ \text{and}~ o(a^i)~ \text{is~ odd}\}$, $C = \{a^{2i}b : ~1 \leq i \leq n \}$ and $D = \{a^{2i + 1}b : ~1 \leq i \leq n \}$. Also, we have discussed the conjugacy classes of the generalized quaternion group group $Q_{4n}$ in the proof of Corollary \ref{conjugacySE-Q_4n} which are
\begin{itemize}
\item $[e]_c = \{e\}; [a^n]_c = \{a^n\}$;
\item for $0 < i < 2n$ and $i \neq n$, we have  $[a^i]_c = \{a^{i}, a^{-i}\}$;
\item $[a^2b]_c = \{a^{2i}b : 0 \leq i < n\}$;
\item $[ab]_c = \{a^{2i + 1}b : 0 \leq i < n\}$.
\end{itemize}
Therefore, the graph $\mathcal{P}^c(Q_{4n})$ is given in Figure \ref{Fig-7}.
Let $S_1 = C$ and $S_2 = D$ be subsets of a vertex set $Q_{4n}$. In view of Figure \ref{Fig-7}, the vertices belonging to the set $S_1$ and $S_2$ form a clique of size $n$, respectively. Also $N[x] \setminus S_1 = \{e, a^n\}$ for all $x \in S_1$ and $N[x] \setminus S_2 = \{e, a^n\}$ for all $x \in S_2$.  By Theorem \ref{multiplicity}, we have $-\sqrt{2} (n+1)$ is an eigenvalue of multiplicity of at least $2n-2$.

\noindent (ii) For odd $n$, we get $n + 2$ is odd. By Equation \ref{eq(2)}, we have $a^2b \sim a^{n+2}b$ in  $\mathcal{P}^c(Q_{4n})$ so that the subgraph induced by the vertices belongs to the union of the sets $C$ and $D$ defined in part (i) is complete. Furthermore, note that the remaining edges will be the same in both cases (even $n$ or odd $n$). Thus, we have the graph $\mathcal{P}^c(Q_{4n})$ is shown in Figure \ref{Fig-8}. By using the similar argument used in part (i), we get the result.
\begin{figure}[h!]
\begin{center}
\begin{tikzpicture}[scale=1.5]
\draw[color=red] (0,2) circle [radius=1.5];
\draw (0,-1.6) node{e};

\draw (-0.6,2.1) node{\textbf{A}};
\draw (0.7,2.1) node{\textbf{B}};
\draw (-1.1,2.1)--(-2.9,2.1);
\draw (-3.1,2.2) node{$a^n$};
\draw (-2.9,2.1)--(0,-1.5);
\draw (0,-1.5)--(0,0.5);



\draw[color=red] (-0.6, 2.1) circle [radius=0.5];

\draw[color=black] (-4.6, -0.1) circle [radius=0.7];
\draw (-4.6, -0.8)--(0,-1.5);
\draw (-4.6, 0.6)--(-2.9,2.1);
\draw (-4.6, -0.1) node{$C \cup D$};


%
\draw[color=red] (0.7,2.1) circle [radius=0.5];
\end{tikzpicture}
\caption{Power graph $\mathcal{P}^c(Q_{4n})$, whenever $n$ is odd}\label{Fig-8}
\end{center}
\end{figure}
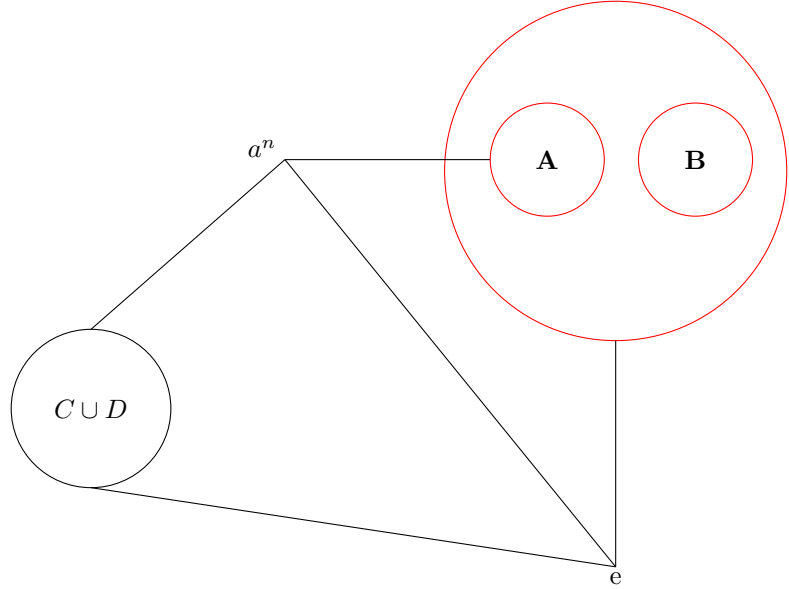
\end{proof}

\begin{theorem}\label{power SD_8n}
 Let $G$ be a semidihedral  group $SD_{8n}$.Then we have the following statements.
\begin{enumerate}
\item[\rm (i)]For odd $n$, $-(2n+1)\sqrt{2}$ and  $-n\sqrt{2}$  are Sombor eigenvalues of power graph $\mathcal{P}^c(SD_{8n})$ and both have multiplicity $2n-1$ and $2n-2$, respectively. 

\item[\rm (ii)] For even $n$, $-(2n)\sqrt{2}$ and  $-(2n+1)\sqrt{2}$ are Sombor eigenvalues of the power graph $\mathcal{P}^c(SD_{8n})$ with multiplicity $2n-1$ each.
\end{enumerate}  
\end{theorem}

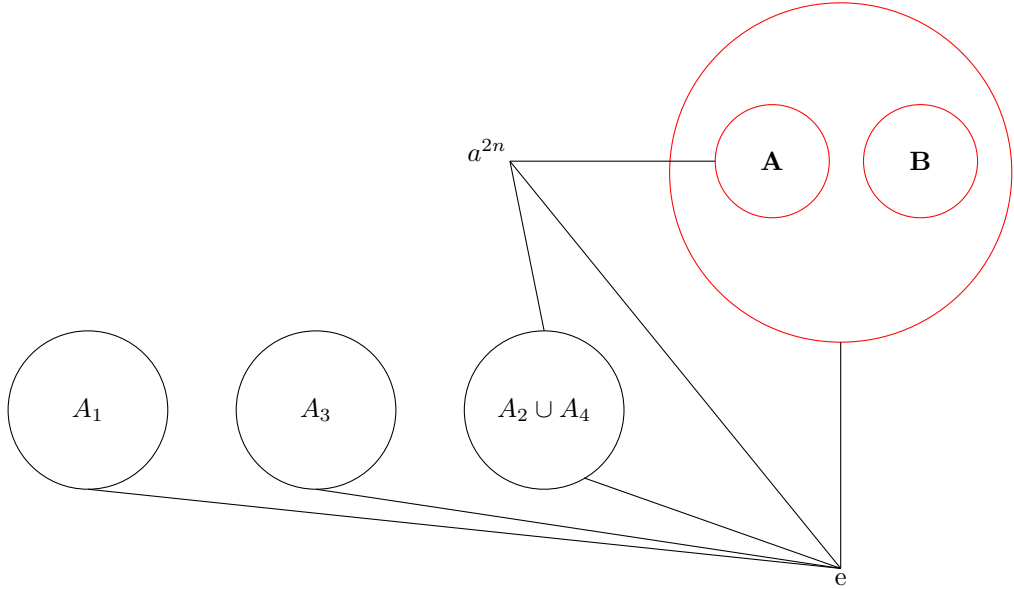
\begin{figure}[h!]
\begin{center}
\begin{tikzpicture}[scale=1.5]
\draw[color=red] (0,2) circle [radius=1.5];
\draw (0,-1.6) node{e};

\draw (-0.6,2.1) node{\textbf{A}};
\draw (0.7,2.1) node{\textbf{B}};
\draw (-1.1,2.1)--(-2.9,2.1);
\draw (-3.1,2.2) node{$a^{2n}$};
\draw (-2.9,2.1)--(0,-1.5);
\draw (0,-1.5)--(0,0.5);




\draw[color=red] (-0.6, 2.1) circle [radius=0.5];
\draw[color=black] (-2.6, -0.1) circle [radius=0.7];
\draw (-2.25, -0.7)--(0,-1.5);
\draw (-2.6, -0.1) node{$A_2 \cup A_4$};
\draw (-2.9,2.1)--(-2.6, 0.6);
\draw[color=black] (-4.6, -0.1) circle [radius=0.7];
\draw (-4.6, -0.8)--(0,-1.5);


\draw (-4.6, -0.1) node{$A_3$};

\draw[color=red] (0.7,2.1) circle [radius=0.5];
\draw (-6.6, -0.1) node{$A_1$};

\draw[color=black] (-6.6, -0.1) circle [radius=0.7];
\draw (0,-1.5)--(-6.6, -0.8);
\end{tikzpicture}
\caption{Power graph $\mathcal{P}^c(SD_{8n})$, when $n$ is odd}\label{Fig-9}
\end{center}
\end{figure}

\begin{proof} (i) If $n$ is odd, then in order to obtain the structure of conjugacy super power graph $\mathcal{P}^c(SD_{8n})$, it is known that the conjugacy classes of semidihedral group $SD_{8n}$ are
\begin{itemize}
\item $[e]_c = \{e\}$;
\item $[a^{n}]_c = \{a^{n}\}$;
\item $[a^{2n}]_c = \{a^{2n}\}$;
\item $[a^{3n}]_c = \{a^{3n}\}$;
\item for $1 \leq i \leq 2n + 1$ and $i \neq \{n, 2n-1, 2n\}$, we have $[a^i]_c = \begin{dcases}
\{a^i, a^{4n -i}\},& \text{ if $i$ is even};\\
\{a^i, a^{2n -i}\},& \text{if $i$ is odd};
\end{dcases}$

\item $[b]_c = \{a^{4k}b:~ 0 \leq k \leq n-1 \}$;

\item $[ab]_c = \{a^{4k + 1}b:~ 0 \leq k \leq n-1 \}$;

\item $[a^2b]_c = \{a^{4k + 2}b:~ 0 \leq k \leq n-1 \}$;

\item $[a^3b]_c = \{a^{4k + 3}b:~ 0 \leq k \leq n-1 \}$.
\end{itemize}

Now we consider the sets $A = \{ a^i : 1 \leq i \leq 4n -1,~ i \neq n ~ \text{and}~ o(a^i)~ \text{is~ even}\}$, $B = \{ a^i : 1 \leq i \leq 4n -1,~ i \neq n ~ \text{and}~ o(a^i)~ \text{is~ odd}\}$, $A_1 = [b]_c = \{a^{4k}b:~ 0 \leq k \leq n-1 \}, A_2 = [ab]_c = \{a^{4k + 1}b:~ 0 \leq k \leq n-1 \}, A_3 = [a^2b]_c = \{a^{4k + 2}b:~ 0 \leq k \leq n-1 \}$ and $A_4 = [a^3b]_c = \{a^{4k + 3}b:~ 0 \leq k \leq n-1 \}$. Now, $ab \in A_2$ and $ab$ is adjacent to $a^{2n+1}b \in A_4$ in $\mathcal{P}(SD_{8n})$. In view of power graph $\mathcal{P}(SD_{8n})$ given in Figure \ref{Fig-2} and the conjugacy classes of semidihedral group $SD_{8n}$, the graph $\mathcal{P}^c(SD_{8n})$ is shown in Figure \ref{Fig-9}.

Let $S_1 = A_1$, $S_2 = A_3$, and $S_3 = A_2 \cup A_4$ be subsets of a vertex set $SD_{8n}$. Clearly, the sets $S_1$, $S_2$, and $S_3$ form cliques of size $n$, $n$, and $2n$, respectively. Moreover, $N[x] \setminus S_i = \{e\}$ for $x \in S_i$ for $1 \leq i \leq 2$. Also,  $N[x] \setminus S_3 = \{e, a^{2n}\}$.  By Lemma \ref{multiplicity}, we have $-(2n + 1)\sqrt{2}$ and $-n\sqrt{2}$ are eigenvalues of multiplicity of at least $2n-1$ and $2n-2$, respectively.\\
(ii) For even $n$, in order to obtain the structure of conjugacy super power graph $\mathcal{P}^c(SD_{8n})$, it is known that the conjugacy classes of semidihedral group $SD_{8n}$ are
\begin{itemize}
\item $[e]_c = \{e\}$;
\item $[a^{2n}]_c = \{a^{2n}\}$;
\item for $1 \leq i \leq 2n + 1$ and $i \neq \{n, 2n-1, 2n\}$, we have $[a^i]_c = \begin{dcases}
\{a^i, a^{4n -i}\},& \text{ if $i$ is even};\\
\{a^i, a^{2n -i}\},& \text{if $i$ is odd};
\end{dcases}$

\item $[b]_c = \{a^{2k}b:~ 0 \leq k \leq 2n-1 \}$;

\item $[ab]_c = \{a^{2k + 1}b:~ 0 \leq k \leq 2n-1 \}$.
\end{itemize}
Now consider the sets $A = \{ a^i : 1 \leq i \leq 4n -1,~ i \neq n ~ \text{and}~ o(a^i)~ \text{is~ even}\}$, $B = \{ a^i : 1 \leq i \leq 4n -1,~ i \neq n ~ \text{and}~ o(a^i)~ \text{is~ odd}\}$, $A_1 = [b]_c = \{a^{2k}b:~ 0 \leq k \leq 2n-1 \}, A_2 = [ab]_c = \{a^{2k + 1}b:~ 0 \leq k \leq 2n-1 \}$.
For $0 \leq i,j \leq 2n-1$, we observe that $a^{2i}b \nsim a^{2j+1}b$ in $\mathcal{P}(SD_{8n})$. In addition, we notice that the remaining edges will be the same in both cases (even $n$ or odd $n$). Thus, we have the graph $\mathcal{P}^c(SD_{8n})$ is shown in Figure \ref{Fig-10}. By using a similar argument used in part (i), we get the result.

\begin{figure}[h!]
\begin{center}
\begin{tikzpicture}[scale=1.5]
\draw[color=red] (0,2) circle [radius=1.5];
\draw (0,-1.6) node{e};

\draw (-0.6,2.1) node{\textbf{A}};
\draw (0.7,2.1) node{\textbf{B}};
\draw (-1.1,2.1)--(-2.9,2.1);
\draw (-3.1,2.2) node{$a^{2n}$};
\draw (-2.9,2.1)--(0,-1.5);
\draw (0,-1.5)--(0,0.5);




\draw[color=red] (-0.6, 2.1) circle [radius=0.5];
\draw[color=black] (-2.6, -0.1) circle [radius=0.7];
\draw (-2.25, -0.7)--(0,-1.5);
\draw (-2.6, -0.1) node{$A_1$};

\draw[color=black] (-4.6, -0.1) circle [radius=0.7];
\draw (-4.6, -0.8)--(0,-1.5);
\draw (-4.6, 0.6)--(-2.9,2.1);



\draw (-4.6, -0.1) node{$A_2$};

\draw[color=red] (0.7,2.1) circle [radius=0.5];
\end{tikzpicture}
\caption{Power graph $\mathcal{P}^c(SD_{8n})$, when $n$ is even}\label{Fig-10}
\end{center}
\end{figure}
\end{proof}

\textbf{Acknowledgement:} The first author gratefully acknowledges for providing financial support to CSIR\\
(09/0719(17365)/2024-EMR-I), Government of India. The third author wishes to acknowledge the support of Grant (CRG/2022/001142) funded by ANRF(SERB), Government of India.

\noindent \textbf{Conflict of Interest:} On behalf of all authors, the corresponding author declares that there is no conflict of interest.


\end{document}